\documentclass[letterpaper,11pt]{article}
\usepackage[sc]{mathpazo}
\usepackage{microtype}
\usepackage[utf8]{inputenc}
\usepackage{graphicx,fullpage,paralist}
\usepackage{amsmath, amssymb, amsthm}
\usepackage{comment,hyperref}
\usepackage{thmtools}
\usepackage{tikz}
\usepackage{pgfplots}
\usepackage{varwidth}
\pgfplotsset{compat=1.10}
\usepgfplotslibrary{fillbetween}
\usetikzlibrary{backgrounds}
\usetikzlibrary{patterns}
\usepackage{enumitem}
\usepackage{geometry}
\usepackage{array}
\usepackage{multirow}
\usepackage{stmaryrd}
\usepackage[font=small]{caption}

\newcommand{\calR}{\mathcal{R}}
\newcommand{\calD}{\mathcal{D}}
\newcommand{\calB}{\mathcal{B}}
\newcommand{\calW}{\mathcal{W}}
\newcommand{\calX}{\mathcal{X}}

\newcommand{\calG}{\mathcal{G}}
\newcommand{\calV}{\mathcal{V}}
\newcommand{\calS}{\mathcal{S}}
\newcommand{\calH}{\mathcal{H}}
\newcommand{\calM}{\mathcal{M}}

\newcommand{\calJ}{\mathcal{J}}
\newcommand{\calU}{\mathcal{U}}
\newcommand{\calI}{\mathcal{I}}
\newcommand{\calP}{\mathcal{P}}
\newcommand{\calQ}{\mathcal{Q}}
\newcommand{\calF}{\mathcal{F}}
\newcommand{\calC}{\mathcal{C}}

\newcommand{\calA}{\mathcal{A}}
\newcommand{\calT}{\mathcal{T}}
\newcommand{\calE}{\mathcal{E}}
\newcommand{\calZ}{\mathcal{Z}}

\newcommand{\RR}{\mathbb{R}}
\newcommand{\NN}{\mathbb{N}}
\newcommand{\ZZ}{\mathbb{Z}}

\newcommand{\prefix}{\mathsf{Prefix}}
\newcommand{\rank}{\mathsf{Rank}_{\calI}}

\newcommand{\apportiong}{\calA_{\gamma}}

\newcommand{\iparty}{\calQ(\mathcal{I})}
\newcommand{\party}{\mathsf{party}}
\newcommand{\votes}{\mathsf{votes}}
\newcommand{\type}{\mathsf{type}}
\newcommand{\ftype}{\mathsf{f}}
\newcommand{\pparty}{\mathsf{p}}
\newcommand{\qparty}{\mathsf{q}}
\newcommand{\gtype}{\mathsf{g}}
\newcommand{\mtype}{\mathsf{m}}

\newtheorem{theorem}{Theorem}
\newtheorem{lemma}{Lemma}
\newtheorem{corollary}{Corollary}

\newtheorem{proposition}{Proposition}
\newtheorem{definition}{Definition}

\newtheorem{claim}{Claim}

\usepackage[textsize=small,textwidth=2cm]{todonotes}
\usepackage{xparse}
\makeatletter
\NewDocumentCommand{\LeftComment}{s m}{%
  \Statex \IfBooleanF{#1}{\hspace*{\ALG@thistlm}}\(\triangleright\) #2}
\makeatother

\usepackage{multirow}
\usepackage{algorithm,algorithmicx}
\usepackage[noend]{algpseudocode}
\usepackage{xcolor}

\newlength{\algofontsize}
\hypersetup{
	colorlinks,
	linkcolor={red!50!black},
	citecolor={blue!50!black},
	urlcolor={blue!80!black}
}

\begin{document}
	\algrenewcommand\algorithmicrequire{\textbf{Input:}}
	\algrenewcommand\algorithmicensure{\textbf{Output:}}
	
	\title{Apportionment with Parity Constraints}
	
	\author{Claire Mathieu
	\thanks{CNRS, IRIF,
		Université de Paris, France. {\tt claire.mathieu@irif.fr}}
	\and 	Victor Verdugo		
	\thanks{Institute of Engineering Sciences, Universidad de O'Higgins, Chile. {\tt victor.verdugo@uoh.cl}}
	}
\date{\vspace{-1em}}
\maketitle
\begin{abstract}
In the classic apportionment problem the goal is to decide how many seats of a parliament should be allocated to each party as a result of an election.
The divisor methods provide a way of solving this problem by defining a notion of proportionality guided by some rounding rule.
Motivated by recent challenges in the context of electoral apportionment, we consider the question of how to allocate the seats of a parliament under parity constraints between candidate types (e.g. equal number of men and women elected) while at the same time satisfying party proportionality. 

We consider two different approaches for this problem.
The first mechanism, that follows a greedy approach, corresponds to a recent mechanism used in the Chilean Constitutional Convention 2021 election.
We analyze this mechanism from a theoretical point of view.
The second mechanism follows the idea of biproportionality introduced by Balinski and Demange [Math. Program. 1989, Math. Oper. Res. 1989].
In contrast with the classic biproportional method by Balinski and Demange, this mechanism is ruled by two levels of proportionality: Proportionality is satisfied at the level of parties by means of a divisor method, and then biproportionality is used to decide the number of candidates allocated to each type and party.
We provide a theoretical analysis of this mechanism, making progress on the theoretical understanding of methods with two levels of proportionality.
A typical benchmark used in the context of two-dimensional apportionment is the {\it fair share} (a.k.a {\it matrix scaling}), which corresponds to an ideal fractional biproportional solution.
We provide lower bounds on the distance between these two types of solutions, and we explore their  consequences in the context of two-dimensional apportionment.
\end{abstract}

\thispagestyle{empty}
\newpage


\section{Introduction}

How many seats should be allocated to each political party in an election? This basic question plays a fundamental role in the political organization of societies, and its study has a long and rich history.
The idea of proportionality has been at the core of electoral apportionment methods over the last two centuries, and this notion is captured by the {\it divisor methods}.
These methods are specified by rounding rules with the goal of assigning to each political party a number of seats that is proportional to the number of votes they get.
In this work we study the setting of an electoral process where, apart from the typical proportionality condition, it is required to achieve parity between the representatives of two parts of the population.
Recently, the Chilean Constitutional Convention election of 2021 implemented a new electoral mechanism with the goal of achieving these two goals: Political proportionality and gender parity.  
In this work we analyze the mechanism used for this election, and we also analyze other mechanism based on {\it biproportionality}, a notion introduced by Balinski and Demange that extends the classic divisor methods to a two-dimensional setting \cite{balinski1989mathprog,BalinskiDemange1989}.  

\subsection{A Greedy Apportionment Mechanism}

In October 25, 2020, there was a 2-question referendum in Chile to approve the design of a new Political Constitution.
The first question was: {\it Do you want a new Political Constitution?} and the possible answers were {\it yes} and {\it no}.
The winning option was {\it yes} with nearly 80\% of the total number of votes.
The election for the Constitutional Convention, which is the representative body in charge of writing the new constitution, happened in May 15-16, 2021.
The law approved in March 2020 establishes a mechanism by which the seats of the convention in each political district of the country have to be allocated in order to achieve gender parity and party proportionality~\cite{ley2020,ley2015}.
In what follows we call this method the {\it greedy \& parity correction mechanism}.
Our first goal is to present, formalize and analyze this algorithm, to obtain structural properties of its output.

In our first result we analyze this mechanism from a structural point of view, showing that it satisfies a set of desired properties for an apportionment method.
Then, we study this mechanism from an optimality point of view.
That is, can we show that the greedy \& parity correction mechanism is optimal for some natural objective? We provide a positive answer for this question by showing that this mechanism computes an optimal solution for a certain natural objective, where the goal is to elect a set of representatives that maximize the number of votes obtained, while at the same time achieving proportionality at the party level.
The analysis of the mechanism can be found in Section~\ref{sec:chilean-algorithm}.

\subsection{A Biproportional Apportionment Mechanism}

The notion of proportionality and divisor methods can be extended to the case in which we have two dimensions that rule the apportionment instance.
In the setting of the present paper, the parties represent one dimension of the problem and the types represent a second dimension of the problem.
Balinski and Demange formally defined the notion  of biproportionality for integral solutions and gave a complete characterization~\cite{BalinskiDemange1989,balinski1989mathprog}.
We analyze a method for the apportionment problem with type parity based on biproportionality.
We call this mechanism the {\it biproportional parity mechanism}.
We analyze this mechanism from a structural an axiomatic point of view.
The formal description of the mechanism and its analysis can be found in Section~\ref{sec:biprop-algorithm}.

This mechanism can be seen as a two stage optimization problem in which the first level ensures party proportionality, and then the second level guarantees biproportionality.
We remark that the number of seats for each party are computed from the instance and they are not part of the input, which contrasts with the typical biproportional method where the voting matrix and the seats distributions are fixed and part of the input~\cite{Pukk2017,BalinskiDemange1989}.
This interaction of two levels of apportionment requires a careful treatment,
and we make progress on one of the research directions mentioned by Demange~\cite{demange2013allocating} on the theoretical understanding of the interaction between these two levels of proportionality.

\subsection{Comparing the Fair Share and the Biproportional Solution}

If one relaxes the integrality condition in the biproportional setting, then such a solution is known as {\it fair share} or {\it matrix scaling}, an object that has been studied extensively in the optimization, statistics and algorithms communities.
This solution is not directly implementable in the electoral context since they generally look for integral solutions, but the fair share is used as a benchmark in order to evaluate the quality of a two-dimensional apportionment \cite{maier2010divisor}.
Then, the following question arises naturally: How far are the biproportional solution and the fair share?
We provide lower bounds to answer for this question.
We show that for every integer  value $\ell$ and every rounding rule, we can find an instance where the $\ell_1$ distance between the biproportional solution and the fair share is at least $\ell$.
A second metric of interest is given by the fraction of rows in which the biproportional solution does not respect the rounded fair share. 
We show that this value is in general bounded away from zero.
We provide this analysis in Section \ref{sec:fairness}.

In the one-dimensional case, the fair share is the fractional solution obtained by assigning to each party a number of seats proportional to the number of votes.
Balinski and Young proved that the Jefferson method is the unique divisor method that never violates the fair share rounded down, while the Adams method is the unique that never violates the fair share rounded up~\cite{BYbook}.
Our lower bounds state that these properties are broken in the two-dimensional setting.
Finally, in Section \ref{sec:consequences} we provide some consequences of our results for the quality of apportionment, including an impossibility result for two-dimensional apportionment.
We also compare the two apportionment mechanisms presented in terms of the fair share.
\subsection{Literature Overview}

The divisor methods are widely used at national and regional level and we refer to the book by Balinski and Young \cite{BYbook} and the recent book by Pukelsheim \cite{Pukk2017} for a deep treatment of the theory and use of these methods.
Rote and Zachariasen \cite{RoteZachariazen2007} and later Gaffke and Pukelsheim \cite{gaffke2008vector,gaffke2008divisor} provided a network flow approach for computing a biproportional solution.
Recently, Cembrano et al. \cite{cembrano2021multidimensional} studied the multidimensional generalization of the biproportional method, and provided existence, complexity and algorithmic results.
Other two-dimensional approaches have been considered, most of them based on network flow techniques \cite{pukelsheim2012network,serafini2012parametric,gassner1991biproportional,cox1982controlled}.
We refer to Maier et al. for a real-life benchmark study for the biproportional method \cite{maier2010divisor}.

The basic algorithmic idea behind many algorithms to find a fair share or matrix scaling is based on the Sinkhorn alternate scaling method \cite{sinkhorn1964relationship}.
Balinski and Demange \cite{balinski1989mathprog} proposed an algorithm for matrix scaling based on this idea and 
their algorithm was later analyzed in terms of computational complexity \cite{kalantari2008complexity}.
Remarkable are the algorithms provided by Kalantari and Khachiyan \cite{kalantari1996complexity} and Nemirovski and Rothblum \cite{nemirovski1999complexity}.
Recently, there has been progress on designing faster algorithms for matrix scaling, motivated by its application on machine learning and the analysis of large data sets \cite{chakrabarty2020better,allen2017much,cohen2017matrix}.
We refer to the survey by Idel for a treatment of the theory, applications, algorithms and a historical overview of matrix scaling \cite{idel2016review}.

\section{Preliminaries}
\label{sec:preliminaries}

Formally, consider an election with $n$ political parties and a set of candidates $\calC$.
Each candidate  $c\in\calC$ belongs to exactly one of the parties, $\party_{\calI}(c)\in [n]$, the candidate garners a certain number of votes, $\votes_{\calI}(c)\in \ZZ_+$ and  is of one of two types, $\type_{\calI}(c)\in \{\ftype,\mtype\}$.	
Finally, we are given a strictly positive integer number $h$ corresponding to the total number of seats to be allocated, called the house size.
In the {\it apportionment problem with type parity}, 
we have to find an allocation of the candidates filling every seat and satisfying the following conditions. 
The first condition is that the total number of seats must be occupied by an equal number of candidates from each type to the extent possible.
That is, each of the $h$ available seats is filled and when $h$ is even the total number of seats occupied for the candidates of type $\ftype$ is equal to the total number of seats occupied by the candidates of type $\mtype$;
when $h$ is odd, the amounts differ by exactly one seat. 
The second condition is that the number of candidates allocated from each party is proportional to the total number of votes obtained by the party, a notion that is formally captured by a divisor method.

\subsection{Divisor Methods and Signpost Sequences}
In the most basic form of the {\it apportionment problem}, the input is given by a pair $(\calQ,h)$ where $\calQ$ is a vector of dimension $n$, encoding the number of votes garnered by each party, and $h$ is a strictly positive integer number representing the house size. 
The output, deciding how many of the $h$ seats are given to each party, is formally described by a vector $\calS$ of dimension $n$ such that $\calS_i$ is a non-negative integer  for every $i\in [n]$ and $\sum_{i\in [n]}\calS_i=h$.
Clearly, such a vector $\calS$ always exists, but usually it is required that the solution $\calS$ satisfies certain desirable properties that restrict the set of feasible solutions.

One of the most studied and widely used methods was stated by Thomas Jefferson in 1792:
For each party $i\in [n]$ let $\calS_i=\calR (\calQ_i/\lambda)$  with $\lambda\in \RR_+$ s.t. $\sum_{i\in [n]}\calR( \calQ_i/\lambda)=h$,
where $\calR$ is the rounding function defined as follows: for every $x\in \RR$ we have that $\calR(x)\in \ZZ$ and $x-1\le \calR(x)\le x$.
Observe that when $x$ is fractional the rounding function coincides with the floor of $x$.
This method belongs to a broader family known as {\it divisor methods}, and they are used in various countries at national and regional level~\cite{BYbook,Pukk2017}.
Their main feature is the fact that they capture the notion of proportionality in the case where allocations have to be integral.
Formally, we say that a function $\gamma:\NN\to \RR_+$ is a {\it signpost sequence} if it satisfies the following properties:
\begin{enumerate}[label=(\alph*)]
	\item $\gamma(0)=0$.\label{divisor1}
	\item For every $n\ge 1$ we have that $\gamma(n)\in [n-1,n]$.\label{divisor2}
	\item If there exists $k\ge 2$ such that $\gamma(k)=k-1$, then $\gamma(\ell)<\ell$ for every $\ell\ge 1$.
	If there exists $k\ge 1$ such that $\gamma(k)=k$, then $\gamma(\ell)>\ell-1$ for every $\ell\ge 2$.\label{divisor3}
\end{enumerate}
In particular, every signpost sequence $\gamma$ is strictly increasing over the strictly positive integers.
Every signpost sequence $\gamma$ has a corresponding {\it rounding function} $\calR_{\gamma}:\RR_+\to \NN$ defined as follows: $\calR_{\gamma}(0)=\{0\}$, $\calR_{\gamma}(t)=\{n\}$ when $t\in (\delta(n),\delta(n+1))$ and $\calR_{\gamma}(t)=\{n-1,n\}$ when $t=\gamma(n)>0$.
For every signpost sequence $\gamma$ we can construct an apportionment method for the instances $(\calQ,h)$ as follows: For each party $i\in [n]$ let $\calS_i=\calR_{\gamma} (\calQ_i/\lambda)$ with $\lambda\in \RR_+$ such that $\sum_{i\in [n]}\calR_{\gamma}( \calQ_i/\lambda)=h$.
We remark that this solution is not necessarily unique, and
we denote by $\calA_{\gamma}(\calQ,h)$ the set of solutions.
The multiplicity of solutions occurs in very specific cases and for most of the instances observed in practice we have that $\calA_{\gamma}(\calQ,h)$ is uniquely defined.
The Jefferson method is the divisor method associated to the signpost sequence $\gamma(n)=n$ for every positive integer $n$.
For a detailed treatment of apportionment theory and divisor methods we refer to the book of Balinski and Young~\cite{BYbook}.
The following lemma summarizes a few properties of the divisor methods that we use in our analysis, and its proof can be found in the Appendix.
\begin{lemma}
	\label{lem:jefferson-one-dim}
	Let $(\calQ,h)$ be an instance of the apportionment problem and let $\gamma$ be a signpost sequence.
	Then, the following holds.
	\begin{enumerate}[label=(\alph*)]
		\item Suppose that $(\calQ,h)$ is such that $\sum_{i\in [n]}\calQ_i/\lambda=h$ for some $\lambda>0$ and such that $\calQ_i/\lambda\in \ZZ$ for every $i\in [n]$.
		Then, $\apportiong(\calQ,h)=\{\calQ/\lambda\}$. \label{jeff-proportional}
		\item For every positive real $\alpha>0$, we have that $\apportiong(\calQ,h)=\apportiong(\alpha\calQ,h)$. \label{jeff-scaling}
		\item Suppose that $(\calQ,h)$ and $(\calQ',h)$ are such that $\calQ_{\pparty}>\calQ'_{\pparty}$ for some party $\pparty\in [n]$ and $\calQ_{i}=\calQ'_{i}$ for every $i\ne \pparty$. 
		Then, for every $\calJ\in \apportiong(\calQ,h)$ and $\calJ'\in \apportiong(\calQ',h)$ we have that $\calJ_{\pparty}\ge \calJ'_{\pparty}$.\label{jeff-monotone}
		\item Suppose that $(\calQ,h)$ and $(\calQ',h)$ are such that $\calQ_{\pparty}>\calQ'_{\pparty}$ for some party $\pparty\in [n]$ and $\calQ_{i}=\calQ'_{i}$ for every $i\ne \pparty$. 
		Then, for every $\calJ\in \apportiong(\calQ,h)$ there exists $\calJ'\in \apportiong(\calQ',h)$ such that the following holds: $\calJ_{\pparty}\ge \calJ'_{\pparty}$ and for every $i\ne \pparty$ we have that $\calJ_{i}\le \calJ'_{i}$. \label{jeff-monotone-strong}
	\end{enumerate}	
\end{lemma}

\subsection{The Apportionment Problem with Type Parity}
An instance with types is represented as a tuple $\calI=(\calC,\party_{\calI},\votes_{\calI},\type_{\calI},h)$.
We denote by $\calC_i$ the set of candidates from party $i$, by $\calC^t$ the set of candidates with type $t$ and $\calC_i^t=\calC_i\cap \calC^t$ are the candidates in the intersection of both.
Observe that to every instance with types $\calI$ we can associate an instance $\iparty=(\calQ,h)$ of the classic apportionment problem by defining 
$\calQ_i=\sum_{c\in \calC_i}\votes_{\calI}(c)$
for each party $i\in [n]$ and by using the same house size $h$ of $\calI$.
Furthermore, for every instance $\calI$ we are given a total order $(\calC,\succ_{\calI})$ over the set of candidates $\calC$ such that the following holds: For every $c,\overline c\in \calC$ we have that $c\succ_{\calI} \overline{c}$ whenever $\votes_{\calI}(c)>\votes_{\calI}(\overline c)$.
For any subset $D\subseteq \calC$, the {\it top candidate} of $D$ is the largest candidate in $D$ according to $\succ_{\calI}$ and for a natural value $k$, the {\it top/worst $k$} candidates of $D$ correspond to the subset of $k$ largest/smallest candidates of $D$.
The scaling of this instance by a positive real $\alpha$ corresponds to the instance $\alpha\calI=(\calC,\party_{\calI},\alpha\cdot \votes_{\calI},\type_{\calI},h)$.
We assume that the ranking of the candidates according to $\succ_{\alpha \calI}$ is the same obtained from $\succ_{\calI}$.

In the apportionment problem with type parity, given an instance with types and a signpost sequence $\gamma$,
we look for an allocation $\calE:\calC\to \{0,1\}$ satisfying the following conditions:
\begin{enumerate}[label=(\Alph*)]
	\item  {\bf Type Parity.} 
	$\Big|\sum_{c\in \calC^{\ftype}}\calE(c)-\sum_{c\in \calC^{\mtype}}\calE(c)\Big|= (h\text{ mod } 2)$.\label{parity-a}
	\item  {\bf Party Proportionality.} 
	For every $i\in [n]$ we have $\sum_{c\in \calC_i}\calE(c)=\calJ_i,$ where $\calJ\in \apportiong(\iparty) $.\label{parity-b}
\end{enumerate}	
We say that $\calE$ is a {\it feasible solution} for $(\calI,\gamma)$ if there exists $\calJ\in \apportiong(\iparty)$ such that $\calE$ satisfies conditions \ref{parity-a} and \ref{parity-b}.  
We denote by $\calX(\calI,\gamma)$ the set of feasible solutions for $(\calI,\gamma)$. 
We say that an instance $\calI$ is {\it feasible for $\gamma$} if $\calX(\calI,\gamma)\ne \emptyset$.
Unfortunately, there are instances and signpost sequences for which is {\bf not} possible to simultaneously achieve both \ref{parity-a} and \ref{parity-b}.
Consider the following instance: We have two parties $\{1,2\}$ and $h=16$ available seats. 
Party 1 has fourteen candidates in total, eight of type $\ftype$ and six of type $\mtype$.
Party 2 has two candidates in total, one of type $\ftype$ and one of type $\mtype$.
Since $h$ is even, any feasible assignment must allocate eight seats to type $\mtype$, but in total we have only seven candidates of that type.
In the following proposition we show that determining the feasibility of an instance can be done by using linear programming.
\begin{proposition}
	\label{prop:linear-program}
	Let $\calI$ be an instance and let $\calJ\in \apportiong(\iparty)$ for some signpost sequence $\gamma$.
	Let $\calZ(\calI,\calJ,\gamma)$ be the polytope of feasible solutions defined by the following linear program:
	\begin{align}
		\sum_{i\in [n]}y_{it}&\ge \lfloor h/2\rfloor \quad \text{ for each }t\in \{\ftype,\mtype\},\label{LPfeasible1}\\
		y_{i\ftype}+y_{i\mtype}&=\calJ_i \quad\quad\;\; \text{ for every }i\in [n],\label{LPfeasible2}\\
		0\le y_{it}&\le |\calC_i^t|\quad\quad \text{ for every }i\in [n]\text{ and each }t\in \{\ftype,\mtype\}. \label{LPfeasible3}
	\end{align}
	Then, there exists a solution satisfying \ref{parity-a}-\ref{parity-b} if and only if $\calZ(\calI,\calJ,\gamma)\ne \emptyset$.
\end{proposition}
\begin{proof}
	
	By total unimodularity it holds that $\calZ(\calI,\calJ,\gamma)=\text{conv}(\{y\in \calZ(\calI,\calJ,\gamma):y\text{ is integral}\})$ and therefore $\calZ(\calI,\calJ,\gamma)\ne \emptyset$ if and only if there exists an integral solution $y\in \calZ(\calI,\calJ,\gamma)$.
	On the other hand, for every integral $y$ consider the solution $\calE_{y}:\calC\to \{0,1\}$ defined as follows: For every $i\in [n]$ and each $t\in \{\ftype,\mtype\}$, let $\calE_y(c)=1$ when $c$ belongs to the top $\min\{y_{it},|\calC_i^t|\}$ candidates of $|\calC_i^t|$, and zero otherwise.
	It holds that $\calE_y$ satisfies \ref{parity-a} and \ref{parity-b} for some $\calJ\in \apportiong(\calQ(\calI))$ if and only if $y\in \calZ(\calI,\calJ,\gamma)$.
\end{proof}

\noindent Of particular interest is the subset of instances that satisfy the next condition:
\begin{enumerate}[label=\text{(SoS.\arabic*)},align=right,leftmargin=4.8\labelwidth]
	\item [\bf (Supply Condition)]$\calI$ satisfies the {\it supply condition} if  $|\calC_i^t|\ge \lceil h/2\rceil $ for each $i\in [n]$ and $t \in \{\ftype,\mtype\}$.\label{supply}
\end{enumerate}
In the following lemma we show that instances satisfying the supply condition are feasible for every divisor method $\gamma$ used in condition \ref{parity-b}.
\begin{lemma}
	\label{lem:closed-supply}
	Let $\calI$ be an instance satisfying the supply condition.
	Then, we have $\calZ(\calI,\calJ,\gamma)\ne \emptyset$ for every $\calJ\in \apportiong(\iparty)$.
	In particular, for every signpost sequence $\gamma$ we have $\calX(\calI,\gamma)\ne \emptyset$.
\end{lemma}

\begin{proof}
	Let $\gamma$ be a signpost sequence and take an instance $\calI$ satisfying the supply condition.
	We show that $\calZ(\calI,\calJ,\gamma)\ne \emptyset$ for every $\calJ\in \apportiong(\iparty)$ when $\calI$ satisfies the supply condition.
	To show that this polytope is not empty, it is sufficient to prove that the dual of the linear program (\ref{LPfeasible1})-(\ref{LPfeasible3}) is not unbounded. 
	The dual is given by the following linear program:
	\begin{align}
		\text{maximize} \quad  \lfloor h/2\rfloor (\beta_{\ftype}&+\beta_{\mtype})+\sum_{i\in [n]}\calJ_i\xi_i-\sum_{i\in [n]}u_{i\ftype}|\calC_i^{\ftype}|-\sum_{i\in [n]}u_{i\mtype}|\calC_i^{\mtype}|\label{dualobj}\\
		\text{subject to} \quad\quad \beta_{\ftype}+\xi_i&\le u_{i\ftype} \;\quad\hspace{-3pt}\text{ for every }i\in [n],\label{dual1}\\
		\beta_{\mtype}+\xi_i&\le u_{i\mtype} \quad\hspace{-3pt}\text{ for every }i\in [n],\label{dual2}\\
		\beta,u&\ge 0.\label{dual3}
	\end{align}
	We show that for any feasible solution $(\xi,\beta,u)$ the objective value is upper bounded by zero, which suffices to imply that $\calZ(\calI,\calJ,\gamma)\ne \emptyset$.
	Observe that constraints (\ref{dual1})-(\ref{dual2}) implies that for every feasible solution $(\xi,\beta,u)$ and every $i\in [n]$ it holds that $\xi_i\le \frac{1}{2}(u_{i\ftype}-\beta_{\ftype}+u_{i\mtype}-\beta_{\mtype}).$
	Then, we have that the objective value of a feasible solution $(\xi,\beta,u)$ can be upper bounded by
	\begin{align*}
		&\lfloor h/2\rfloor (\beta_{\ftype}+\beta_{\mtype})+\frac{1}{2}\sum_{i\in [n]}\calJ_i\Big(u_{i\ftype}-\beta_{\ftype}+u_{i\mtype}-\beta_{\mtype}\Big)-\sum_{i\in [n]}u_{i\ftype}|\calC_i^{\ftype}|-\sum_{i\in [n]}u_{i\mtype}|\calC_i^{\mtype}|\\
		&\le \Big(\lfloor h/2\rfloor-h/2\Big) (\beta_{\ftype}+\beta_{\mtype})+\frac{1}{2}\sum_{i\in [n]}\Big(\calJ_i-2\lceil h/2\rceil\Big)\Big(u_{i\ftype}+u_{i\mtype}\Big)\le 0,
	\end{align*}
	since $\calI$ satisfies the supply condition, $\sum_{i\in [n]}\calJ_i=h$ and the last inequality holds from $\beta,u\ge 0$ and $\calJ_i\le h\le 2\lceil h/2\rceil$ for every $i\in [n]$.
	Proposition \ref{prop:linear-program} implies that for every signpost sequence $\gamma$ we have that $\calX(\calI,\gamma)\ne \emptyset$.
\end{proof}
\begin{definition}
	Let $\calW$ be a subset of instances and let $\Delta$ be the set of signpost sequences.
	A set valued function $\calM$ from $\calW\times \Delta$ is a valid apportionment mechanism over $\calW$ if for every signpost sequence $\gamma$ the following holds: $\calM(\calI,\gamma)\subseteq \calX(\calI,\gamma)$  for every $\calI\in \calW$, and $\calM(\calI,\gamma)\ne \emptyset$ when $\calX(\calI,\gamma)\ne \emptyset$.
\end{definition}

Given an instance $\calI$, consider an instance $\calG=(\calC,\party_{\calG},\votes_{\calG},\type_{\calG},h)$ such that the following holds:  $\type_{\calI}=\type_{\calG}$, $\party_{\calI}=\party_{\calG}$, there exists a candidate $c\in \calC$ with $\votes_{\calG}(c)> \votes_{\calI}(c)$ and $\votes_{\calI}(s)= \votes_{\calG}(s)$ for every $s\in \calC$ with $s\ne c$. 
We say that $\calG$ is a {\it voting increment of $\calI$ in $c\in \calC$}.

\begin{definition}
	Let $\calW$ be the subset of instances satisfying the supply condition.
	We say that a set valued function $\calM$ from $\calW\times \Delta$ is {\it $\gamma$-satisfactory over $\calW$} if $\calM$ is a valid apportionment mechanism over $\calW$ and furthermore it satisfies the following properties: 
	\begin{enumerate}[label=(\Roman*), ref=(\Roman*)]
		\item {\bf Exactness.} Let $\calI\in \calW$ be an instance such that $\votes_{\calI}\in \calX(\calI,\gamma)$. 
		Then  $\calM(\calI,\gamma)=\{\votes_{\calI}\}$.\label{axiom2}
		\item {\bf Scaling.} For every instance $\calI\in \calW$, every signpost sequence $\gamma$, and every positive real $\alpha$, we have $\calM(\calI,\gamma)=\calM(\alpha\calI,\gamma)$.\label{axiom3}
		\item {\bf Monotonicity.}
		Let $\calI\in \calW$ be an instance with set of candidates $\calC$ and let $\calG$ and $c\in \calC$ such that $\calG$ is a voting increment of $\calI$ in $c\in \calC$.
		Then, for every $\calE_{\calI}\in \calM(\calI,\gamma)$ there exists $\calE_{\calG}\in \calM(\calG,\gamma)$ such that $\calE_{\calG}(c)\ge \calE_{\calI}(c)$.\label{axiom4}
	\end{enumerate}
\end{definition}

\noindent The first property states that when the votes of the instance is a already a feasible solution, the mechanism must return this as the solution.
The second property states that the solution returned by the mechanism is invariant under votes scaling, and
the third property states that a candidate does not lose a seat as a consequence of garnering at least one additional vote, while the rest of the candidates remains the same.
The following proposition states that feasibility is preserved under scaling and voting increments for the instances satisfying the supply condition. 
\begin{proposition}
	\label{prop:closedness}
	Suppose that $\calI$ is feasible for a signpost sequence $\gamma$.
	Then, for every positive real $\alpha$, we have that $\alpha \calI$ is feasible for a signpost sequence $\gamma$.
	Furthermore, when $\calI$ satisfies the supply condition, we have that every voting increment $\calG$ of $\calI$ is feasible for every signpost sequence $\gamma$.
\end{proposition}
\begin{proof}
	Observe that thanks to Lemma \ref{lem:jefferson-one-dim} we have that $\apportiong(\calQ(\calI))=\apportiong(\calQ(\alpha \calI))$, and therefore for every $\calX(\calI,\gamma)\subseteq \calX(\alpha \calI,\gamma)$.
	Therefore, when $\calI$ is feasible for $\gamma$, we have that $\alpha \calI$ is feasible for $\gamma$ as well.
	When $\calI$ satisfies the supply condition, we have that any voting increment $\calG$ satisfies the supply condition as well, and therefore by Lemma \ref{lem:closed-supply} we have that $\calX(\calG,\gamma)\ne \emptyset$ for every $\gamma$.
\end{proof}

\section{The Greedy \& Parity Correction Mechanism}
\label{sec:chilean-algorithm}

In this section we describe and analyze the greedy \& parity correction mechanism.
Before presenting the mechanism we need the following definition.

\begin{definition}
	Given an instance $\calI$ and $\calJ\in \apportiong(\calQ(\calI))$ we define the {\it type oblivious solution} $\calT_{\calJ}$ as follows: For every party $i\in [n]$ let $\calT_{\calJ}(c)=1$ for the top $\calJ_i$ candidates of $\calC_i$ and let $\calT_{\calJ}(s)=0$ otherwise.
\end{definition}
That is, the type oblivious solution allocates as many of the top candidates of each party as required according to $\calJ$,  regardless of the candidate types.
Thus it satisfies the party proportionality property \ref{parity-b}, but in general  it does not satisfythe type parity property \ref{parity-a}.
The mechanism performs two phases: In the first phase, it computes a type oblivious solution.
When type parity is not satisfied by the type oblivious solution, there is a second phase in which candidates from the over-represented type, from worst to top on the ranking, are replaced by a top available candidate from the same party and from the under-represented type.
The full description of this procedure can be found in Algorithm \ref{alg:chilean}.

\setcounter{algorithm}{-1}
\begin{algorithm}[H]
	\caption{The Greedy \& Parity Correction Mechanism}
	\begin{algorithmic}[1]
		\Require{An instance $\calI$ satisfying the supply condition and $\calJ\in \apportiong(\iparty)$.}
		\Ensure{A  solution $\calE_{\calJ}$ satisfying condition \ref{parity-a} and \ref{parity-b}.}
		\LeftComment{{\bf Phase 1: Greedy}.}
		\vspace{.1cm}
		\State Compute the type oblivious solution $\calT_{\calJ}$. 
		\If{$|\sum_{c\in \calC^{\mtype}}\calT_{\calJ}(c)-\sum_{c\in \calC^{\ftype}}\calT_{\calJ}(c)|= (h\text{ mod } 2)$} we {\bf stop} and return $\calT_{\calJ}$.
		\EndIf		
		\State Otherwise, let $t^{\star}$ be the over represented type in $\calT_{\calJ}$ and let $t_{\star}$ be the under represented type; continue to Phase 2 and initialize $\calE\leftarrow \calT_{\calJ}$.
		\vspace{.1cm}
		\LeftComment{{\bf Phase 2: Parity Correction}.}
		\vspace{.1cm}
		\While{$|\sum_{c\in \calC^{\mtype}}\calE(c)-\sum_{c\in \calC^{\ftype}}\calE(c)|\neq (h\text{ mod } 2)$}
		\State Let $s\in \calC^{t^{\star}}$ be the worst ranked in $\{c\in \calC^{t^{\star}}:\calE(c)=1\},$ and let $i\in [n]$ be such that $c\in \calC_{i}$.
		\State Let $\overline s\in \calC_i^{t_{\star}}$ be the best ranked in $\{c\in \calC_i^{t_{\star}}:\calE(c)=0\}$ when this subset is not empty. 
		\State We update the allocation: $\calE(s)\leftarrow 0$ and $\calE(\overline s)\leftarrow 1$.
		That is, $\overline s$ replaces $s$ in the allocation.
		\EndWhile
		\State Return $\calE_{\calJ}= \calE$.
	\end{algorithmic}
	\caption{\small Greedy \& Parity Correction \label{alg:chilean}}
\end{algorithm}
The {\it greedy \& parity correction mechanism}, denoted by $\calM^{G}$, is defined as follows: For every pair $(\calI,\gamma)$ we have that $\calM^{G}(\calI,\gamma)=\{\calE_{\calJ}:\calJ\in \apportiong(\iparty)\}$ where $\calE_{\calJ}$ is the solution returned by Algorithm \ref{alg:chilean}.
We remark that when the set $\apportiong(\iparty)$ is just a singleton, the apportionment mechanism will be defined uniquely for the instance.
We have observed that the supply condition guarantees the feasibility of an instance (Lemma \ref{lem:closed-supply})
and now we show it is also necessary for Algorithm \ref{alg:chilean} to terminate with a solution. 
For that, consider the following instance $\calI$ with two parties $\{1,2\}$ and $h=8$ available seats, where each party has three candidates of each type.
The information per party is summarized below.\\
	
	\begin{minipage}[c]{0.5\linewidth}
		\small 
		\begin{tabular}{|c|c|c|c|c|c|c|}
			\hline
			Party 1& $c_1$ & $c_2$ & $c_3$ & $c_4$ & $c_5$ & $c_6$\\
			$\votes_{\calI}$& $4$ & $3$ & $1$ & $165$ & $164$ & $163$\\
			$\type_{\calI}$& $\mtype$ & $\mtype$ & $\mtype$ & $\ftype$ & $\ftype$ & $\ftype$\\
			\hline
		\end{tabular}
	\end{minipage}
	\begin{minipage}[c]{0.5\linewidth}
		\small 
		\begin{tabular}{|c|c|c|c|c|c|c|}
			\hline
			Party 2& $c_7$ & $c_8$ & $c_9$ & $c_{10}$ & $c_{11}$ & $c_{12}$\\
			$\votes_{\calI}$& $93$ & $92$ & $91$ & $9$ & $8$ & $7$\\
			$\type_{\calI}$& $\mtype$ & $\mtype$ & $\mtype$ & $\ftype$ & $\ftype$ & $\ftype$\\
			\hline
		\end{tabular}
	\end{minipage}
\vspace{.5cm}

Since $h$ is even, four candidates of each type have to be elected. 
Observe that the total number of votes obtained by party 1 is 500 and the total number of votes obtained by party 2 is 300.
Since $h=8$, by Lemma \ref{lem:jefferson-one-dim} \ref{jeff-proportional} we have that for every signpost sequence it holds that $\apportiong(\iparty)=\{(5,3)\}$.
Phase 1 of the Algorithm~\ref{alg:chilean} yields the five seats for party 1 and three seats for party 2.
Therefore, at the end of Phase 1, we have that $\calE(c_j)=1$ for every $j\in [6]\setminus \{3\}$ (party 1) and $\calE(c_j)=1$ for $j\in \{7,8,9\}$ (party 2). 
In total we have five candidates of type $\mtype$ and three of type $\ftype$.

In the first iteration of Phase 2 we select the worst candidate that is currently elected and of type $\mtype$, which is $c_2$, but when the algorithm tries to replace it we have that the pool of candidates $\{s\in \calC_1^{\ftype}:\calE(s)=0\}$ is empty.
Therefore, the algorithm is not able to terminate with a solution meeting the requirements.
In contrast, this instance is feasible for every signpost sequence $\gamma$: It is sufficient to update the solution $\calE$ by doing $\calE(c_9)=0$ and $\calE(c_{10})=1$.
The following is our first result regarding this mechanism.

\begin{theorem}
	\label{thm:chilean-satisfactory}
	For every signpost sequence $\gamma$ we have that $\calM^{G}$ is $\gamma$-satisfactory over the  instances satisfying the supply condition.
\end{theorem}

Before proving the theorem, in the following lemma we prove that the greedy \& parity correction mechanism is a valid apportionment mechanism.

\begin{lemma}
\label{lem:chilean-correct}
Let $\calI$ be an instance satisfying the supply condition and consider $\calJ\in \apportiong(\iparty)$ for some signpost sequence $\gamma$.
Then, the solution $\calE_{\calJ}$ computed by Algorithm \ref{alg:chilean} satisfies conditions \ref{parity-a} and \ref{parity-b}.  
In particular, the greedy \& parity correction mechanism $\calM^{G}$ is a valid apportionment mechanism over the instances satisfying the supply condition.
\end{lemma}

\begin{proof}
Let $\calI$ be an instance that satisfies the supply condition and let $\gamma$ be a signpost sequence.
By Lemma \ref{lem:closed-supply} we have that $\calX(\calI,\gamma)\ne \emptyset$ and therefore it is sufficient to show that $\calM^G(\calI,\gamma)\subseteq \calX(\calI,\gamma)$.
Upon completion of Phase 1, the party proportionality condition \ref{parity-b} is satisfied, and  this condition is preserved is preserved throughout Phase 2. During Phase 2 it holds that the unbalance 
$|\sum_{c\in \calC^{\ftype}}\calE(c)-\sum_{c\in \calC^{\mtype}}\calE(c)| $ decreases by two with every swap, and thanks to the supply condition, there are always enough candidates of each party to do the reassignments of Phase 2 without getting stuck, so the algorithm terminates when the unbalance is equal to zero or one, and at that point the type parity condition \ref{parity-a} is satisfied. \end{proof}

\begin{proof}[Proof of Theorem~\ref{thm:chilean-satisfactory}]
By Lemma \ref{lem:chilean-correct} we have that the greedy \& parity correction mechanism is a valid apportionment mechanism over the instances satisfying the supply condition.
Therefore, it remains to check that mechanism satisfies the properties \ref{axiom2}-\ref{axiom3}-\ref{axiom4}.
Recall that for every instance $\calI$ we denote by $\calP(\calI)$ the matrix with entries in $[n]\times \{\ftype,\mtype\}$ such that $\calP_{it}(\calI)=\sum_{c\in \calC_{i}^t}\votes_{\calI}(c)$.\\

\noindent{\bf Exactness.} Now suppose that we are given an instance where the function $\votes_{\calI}$ is such that $\votes_{\calI}(c)\in \{0,1\}$ for every candidate $c\in \calC$ and the function $\votes_{\calI}$ satisfies \ref{parity-a} and \ref{parity-b}.
That is, there are exactly $h$ votes and exactly $h$ candidates obtain exactly one vote each. 
Since $\votes_{\calI}$ satisfies \ref{parity-a}, we have that $\sum_{c\in \calC}\votes_{\calI}(c)=h=\sum_{i\in [n]}(\calP_{i\ftype}(\calI)+\calP_{i\mtype}(\calI))$ and
therefore, we have that 
$\calQ_i(\calI)=\calP_{i\ftype}(\calI)+\calP_{i\mtype}(\calI)$ for each party $i\in [n]$ and $\sum_{i\in [n]}\calQ_i(\calI)=h$, 
which by Lemma~\ref{lem:jefferson-one-dim}~\ref{jeff-proportional} with $\lambda=1$ implies that for every signpost sequence $\gamma$ we have $\apportiong(\iparty)=\{\calJ\}$ where $\calJ_i=\calP_{i\ftype}(\calI)+\calP_{i\mtype}(\calI)$ for each party $i\in [n]$.
Therefore, at the end of Phase 1 the candidates selected by the type oblivious solution are exactly those who obtained a vote. 
Since $\votes_{\calI}$ satisfies condition~\ref{parity-b}, parity is already satisfied and therefore \ref{alg:chilean} does not enter Phase 2.
Therefore, in this case we have $\calM^G(\calI,\gamma)=\{\votes_{I}\}$ for every signpost sequence $\gamma$ and property \ref{axiom2} is satisfied.\\

\noindent{\bf Scaling.} 
Let $\alpha$ be a non-negative real and consider the instance $\alpha\calI=(\calC,\party_{\calI},\alpha\cdot \votes_{\calI},\type_{\calI},h)$. 
By Lemma~\ref{lem:jefferson-one-dim}~\ref{jeff-scaling} we have that $\apportiong(\calQ(\alpha\calI))=\apportiong(\iparty)$ for every signpost sequence $\gamma$.
Therefore, for every $\calJ\in \apportiong(\iparty)$ we have that the type oblivious solution for $\calI$ is the same obtained for $\alpha \calI$. 
During Phase 2, since the ranking according to $\succ_{\alpha \calI}$ is the same obatined from $\succ_{\calI}$, the trajectory followed by Algorithm \ref{alg:chilean} for $\alpha\calI$ and $\calJ$ is the same as in $\calI$ and $\calJ$, resulting in the same final solution at the end of the execution. 

\paragraph{Monotonicity.} 
Let $\calI$ and $\calG$ be two instances as described in the statement of property \ref{axiom4}.
We remark that $\calG$ satisfies the supply condition since $\calI$ does. 
Suppose that $\calG$ is a voting increment of $\calI$ in candidate $c$ and let $\pparty$ and $\gtype$ be the party and type of $c$ respectively.
In particular, we have that $\calP_{\pparty \gtype}(\calG)> \calP_{\pparty \gtype}(\calI)$ and $\calP_{it}(\calI)= \calP_{it}(\calG)$ for every pair $(i,t)\ne (\pparty,\gtype)$.
Therefore, by Lemma~\ref{lem:jefferson-one-dim}~\ref{jeff-monotone} it holds that $\calJ_{\pparty}(\calG)\ge \calJ_{\pparty}(\calI)$ for every $\calJ(\calI)\in \apportiong(\calQ(\calI))$ and $\calJ(\calG)\in \apportiong(\calQ(\calG))$.
Therefore, the total number of candidates from party $\pparty$ assigned to a seat at the end of Phase 1 for $\calI$ and $\calJ(\calI)\in \apportiong(\calQ(\calI))$ is larger than the number obtained for $\calG$ and every $\calJ(\calI)\in \apportiong(\calQ(\calG))$.
In Phase 2 the swaps occur between candidates of the same party, and therefore the total number of candidates assigned to a given party remains invariant. 
Since the ranking of the candidate $c$ in the total order $(\calC,\succ_{\calI})$ is at least as good as the ranking in the total order $(\calC,\succ_{\calG})$, we have that if $\calE_{\calG}(c)=1$ at the end of Phase 2, then $\calE_{\calI}(c)=1$ as well.
Therefore property \ref{axiom4} is satisfied.
\end{proof}

\subsection{Optimality Characterization of the Greedy \& Parity Correction Mechanism}

Given $\calI$ with candidates $\calC$, let $c_1,c_2,c_2,\ldots,c_{|\calC|}$ be the candidates in $\calC$ sorted in non-increasing order according to the total order $\succ_{\calI}$ and we denote by $\rank(c)$ the ranking of the candidate $c$ in this order.
In what follows, given $\calE:\calC\to \{0,1\}$, we denote by $\calV(\calE)$ the vector such that $\calV_j(\calE)=\calE(c_j)$ for every $j\in \{1,2,\ldots,|\calC|\}$.
Given two vectors $x,y\in \{0,1\}^{|\calC|}$, the length of the {\it common prefix} between $x$ and $y$ is equal to $\ell$ if $x_j=y_j$ for every $j\in \{1,2,\ldots,\ell\}$ and $x_{\ell+1}\ne y_{\ell+1}$, and we denote this length by $\prefix(x,y)$.
For every pair $(\calI,\gamma)$ and $\calJ\in \apportiong(\iparty)$, let $\calH_{\calJ}(\calI)$ be the set of feasible solutions that maximize the length of the common prefix with $\calT_{\calJ}$, that is,
\[\calH_{\calJ}(\calI)=\text{argmax}\Big\{\prefix(\calV(\calE),\calV(\calT_{\calJ})):\calE\text{ satisfies \ref{parity-a} and \ref{parity-b}}\Big\}.\]

\begin{theorem}
	\label{thm:longest-prefix}
	For every instance $\calI$ that satisfies the supply condition and $\calJ\in \apportiong(\iparty)$, the Algorithm \ref{alg:chilean} computes a solution $\calE_{\calJ}\in \calH_{\calJ}(\calI)$ that maximizes the total number of votes, that is, 
	\[\sum_{c\in \calC}\calE_{\calJ}(c)\cdot \votes_{\calI}(c)\ge \sum_{c\in \calC}\calE(c)\cdot \votes_{\calI}(c)\quad\text{ for every }\calE\in \calH_{\calJ}(\calI).\]
\end{theorem}
That is, among the set of solutions having the longest common prefix with the type oblivious solution, Algorithm \ref{alg:chilean} selects one that maximizes the total number of votes obtained by the selected candidates.
We state two structural lemmas before proving this theorem.

\begin{lemma}
	\label{lem:chilean-cut}
	Let $\calI$ be an instance that satisfies the supply condition. 
	Consider $\calJ\in \apportiong(\calQ(\calI))$ for some signpost sequence $\gamma$ and let $\calE_{\calJ}$ be the solution computed by Algorithm \ref{alg:chilean}.
	Suppose that the type oblivious solution $\calT_{\calJ}$ does not satisfy the type parity condition \ref{parity-a} and consider $\ell=\prefix(\calV(\calE_{\calJ}),\calV(\calT_{\calJ}))<|\calC|$.
	Then, we have that $\calE_{\calJ}(c_{\ell+1})=0$ and $\calT_{\calJ}(c_{\ell+1})=1$.
\end{lemma}

\begin{proof}
By Lemma~\ref{lem:chilean-correct} we have that $\calE_{\calJ}$ satisfies the type parity and the party proportionality conditions \ref{parity-a}-\ref{parity-b}.
Consider the candidate $c_{\ell+1}$ and suppose that $\calT_{\calJ}(c_{\ell+1})=0$.
Then, we have that $\calE_{\calJ}(c_{\ell+1})=1$ and therefore the candidate $c_{\ell+1}$ was included in the solution during the second phase of Algorithm \ref{alg:chilean}, replacing other other candidate $s$ in the same party of $c_{\ell+1}$, from the other type, and such that $\calT_{\calJ}(s)=1$ and $\calE_{\calJ}(s)=0$. 
Since $\ell=\prefix(\calV(\calE_{\calJ}),\calV(\calT_{\calJ}))$ we have that the ranking of $s$ in the order $\succ_{\calI}$ must be larger than $\ell$, but this is a contradiction: Since the type oblivious selected $s$ but it did not select $c_{\ell+1}$ we have that the ranking of $s$ is less than the ranking of $c_{\ell+1}$.
Then, we conclude that $\calT_{\calJ}(c_{\ell+1})=1$ and $\calE_{\calJ}(c_{\ell+1})=0$.
\end{proof}
\begin{lemma}
\label{lem:chilean-over-contained}
Let $\calI$ be an instance that satisfies the supply condition and consider $\calJ\in \apportiong(\calQ(\calI))$ for some signpost sequence $\gamma$.
Suppose that the type oblivious solution $\calT_{\calJ}$ does not satisfy the type parity condition \ref{parity-a} and let $t^{\star}$ be the over represented type in this solution.
Take $\ell=\prefix(\calV(\calE_{\calJ}),\calV(\calT_{\calJ}))<|\calC|$.
Then, for every $c\in \calC$ of type $t^{\star}$ such that $\calE_{\calJ}(c)=1$, we have $\rank(c)\le \ell$.
\end{lemma}

\begin{proof}
By Lemma \ref{lem:chilean-cut} we have that $\calE_{\calJ}(c_{\ell+1})=0$ and $\calT_{\calJ}(c_{\ell+1})=1$.
Since $c_{\ell+1}$ is allocated in the type oblivious solution but it is not allocated in the solution $\calE_{\calJ}$, it means that the candidate $c_{\ell+1}$ belongs to the over represented type $t^{\star}$ in the type oblivious solution.
Since the length of the common prefix between $\calV(\calE_{\calJ})$ and $\calV(\calT_{\calJ})$ is equal to $\ell$, we have that $c_{\ell+1}$ was the last candidate of type $t^{\star}$ that was swapped in the second phase of Algorithm~\ref{alg:chilean} in order to satisfy the type parity condition \ref{parity-a}. 
In particular, we have that $\calE_{\calJ}(c)=0$ for every candidate $c\in \calC$ of type $t^{\star}$ with $\rank(c)>\ell$.
\end{proof}

\begin{proof}[Proof of Theorem~\ref{thm:longest-prefix}]
Let $\calE_{\calJ}$ be the solution computed by Algorithm \ref{alg:chilean} in instance $\calI$ and $\calJ\in \apportiong(\iparty)$.
Consider $\ell=\prefix(\calV(\calE_{\calJ}),\calV(\calT_{\calJ}))$ and suppose there is other solution $\tilde \calE$ satisfying \ref{parity-a}-\ref{parity-b} and such that $\prefix(\calV(\tilde \calE),\calV(\calT_{\calJ}))\ge \ell+1$.
In particular, it means that there is an over represented type in the type oblivious solution $\calT_{\calJ}$ and we denote it by $t^{\star}$.
By Lemma \ref{lem:chilean-over-contained} we have that $\rank(c)\le \ell$ for every $c\in \calC$ of type $t^{\star}$ such that $\calE_{\calJ}(c)=1$ and we denote this set of candidates by $\calA$.
Since the length of the common prefix between $\calV(\tilde \calE)$ and $\calV(\calT_{\calJ}))$ is at least $\ell+1$, we have $\prefix(\calV(\tilde \calE),\calV(\calE_{\calJ}))=\ell$ and therefore $\calA\subseteq \{c\in \calC^{t^{\star}}:\tilde \calE(c)=1\}$.
This implies that $\sum_{c\in \calC^{t^{\star}}}\tilde\calE(c)\ge |\calA|+\tilde \calE(c_{\ell+1})=|\calA|+1$,
but this contradicts the fact that $\tilde \calE$ is a feasible solution that satisfies the type parity condition \ref{parity-a}.
We conclude that $\calE_{\calJ}\in \calH_{\calJ}(\calI)$.
Now suppose that there exists other solution $\beta\in \calH_{\calJ}(\calI)$ such that
\begin{equation}
	\label{ineq:under}
	\sum_{c\in \calC}\calE_{\calJ}(c)\cdot \votes_{\calI}(c)<\sum_{c\in \calC}\beta(c)\cdot \votes_{\calI}(c).
\end{equation}
Let $\calB$ be the subset of candidates such that $\rank(c)\le \ell$.
In particular, we have that $\calE_{\calJ}(c)=\beta(c)$ for every $c\in \calB$ and therefore we have that $\sum_{c\in \calC\setminus \calB}\calE_{\calJ}(c)\cdot \votes_{\calI}(c)<\sum_{c\in \calC\setminus \calB}\beta(c)\cdot \votes_{\calI}(c)$.
By Lemma \ref{lem:chilean-over-contained} we have that $\{c\in \calC^{t^\star}:\calE_{\calJ}(c)=1\}\subseteq \calB$ and since $\calE_{\calJ}(c)=\beta(c)$ for every $c\in \calB$ we conclude that $\{c\in \calC^{t^\star}:\beta(c)=1\}\subseteq \calB$.
This inclusion, together with the strict inequality (\ref{ineq:under}), implies the existence of two candidates $s,\tilde s\in \calC_i$ for some $i\in [n]$ and from the under represented type $t_{\star}$ such that $\calE_{\calJ}(s)=\beta(\tilde s)=0$, $\calE(\tilde s)=\beta(s)=1$ and $\votes_{I}(\tilde s)<\votes_{I}(s)$.
But this contradicts the swapping rule in Phase 2 of Algorithm \ref{alg:chilean}: the candidate $s$ should be included before including $\tilde s$, that is, if $\calE_{\calJ}(\tilde s)=1$ then we necessarily have $\calE_{\calJ}(s)=1$.
This concludes the proof.  
\end{proof}

\section{The Biproportional Parity Mechanism}
\label{sec:biprop-algorithm}

In the following we describe a biproportional mechanism based on the biproportional method introduced by Balinski and Demange \cite{balinski1989mathprog,BalinskiDemange1989}.
We say that $(\calP,\calS,\calJ,\phi)$ is a {\it two-dimensional instance with supply} if both $\calP$ and $\calS$ are integral matrices with entries in $[n]\times \{\ftype,\mtype\}$, $\calJ$ is an $n$ dimensional non-negative integral vector and $\phi$ is a non-negative integral vector with entries in $\{\ftype,\mtype\}$ such that $\sum_{i\in [n]}\calJ_i=\phi_{\ftype}+\phi_{\mtype}$.
The vectors $\calJ$ and $\phi$ are called {\it row and column marginals}, respectively. 
\begin{definition}
	\label{def:biproportional}
	Let $(\calP,\calS,\calJ,\phi)$ be a two-dimensional instance with supply and let $\delta$ be a signpost sequence.
	Given a matrix $x$ with integer entries in $[n]\times \{\ftype,\mtype\}$, a vector $\lambda\in \RR_+^n$ and $\mu=(\mu_{\ftype},\mu_{\mtype})\in \RR_+^{2}$, we say that the triplet $(x,\lambda,\mu)$ is a biproportional solution for $(\calP,\calS,\calJ,\phi)$ with signpost sequence $\delta$ if for each $i\in [n]$ and each $t\in \{\ftype,\mtype\}$ the following holds:
	\begin{align}
		x_{it}&\in \calR_{\delta}(\calP_{it}\lambda_i\mu_t)\text{ when }x_{it}<\calS_{it},\label{biprop1}\\
		x_{i\ftype}+x_{i\mtype}&=\calJ_i,\label{biprop2}\\
		\sum_{i\in [n]}x_{it}&=\phi_t.\label{biprop3}\\
		0\le x_{it}&\le \calS_{it}. \label{biprop4}
	\end{align}
\end{definition}
We denote by $\calB_{\delta}(\calP,\calS,\calJ,\phi)$ the set of integral matrices $x$ such that there exist $\lambda$ and $\mu$ for which $(x,\lambda,\mu)$ is a $\delta$-biproportional solution for $(\calP,\calS,\calJ,\phi)$.
In particular, we refer to (\ref{biprop1}) as the biproportionality condition.

\subsection{A Network Flow Approach}

Given a signpost sequence $\delta$ such that $\delta(1)>0$, we follow the approach based on network flows to compute a biproportional solution $x\in \calB_{\delta}(\calP,\calS,\calJ,\phi)$, based on the idea by Rote and Zachariasen \cite{RoteZachariazen2007} and more recently by Gaffke and Pukelsheim \cite{gaffke2008vector}. 
In our case we consider a capacitated version of the biproportional method, as a result of the upper bound on the value of each entry of $x$ given by the supply matrix.

Consider a graph with $n+4$ vertices given by a source $v$, one vertex $u_i$ for each $i\in [n]$, one vertex $v_t$ for each type $t\in \{\ftype,\mtype\}$ and one sink $\overline v$. 
There is an edge from the source $v$ to every $u_i$ with $i\in [n]$ and with a capacity lower bound equal to $\calJ_i$. 
For every $i\in [n]$ and each type $t\in \{\ftype,\mtype\}$ the graph has $\calS_{it}$ parallel edges $e_{it1},e_{it2},\ldots,e_{it|\calS_{it}|}$ between $u_i$ and $v_t$. 
Those edges have a capacity upper bound equal to one and the the edge $e_{it\ell}$ has a cost of $\log(\delta(\ell)/\calP_{it})$ for each $\ell \in [\calS_{it}]$. 
Finally, there is an edge from both $v_{\ftype}$ and $v_{\mtype}$ to $\overline v$ with a capacity lower bound of $\phi_{\ftype}$ and $\phi_{\mtype}$ respectively.
The associated minimum cost flow problem is the following:
\begin{align}
\text{minimize} \quad\sum_{i=1}^n\sum_{\ell=1}^{|\calS_{i\ftype}|} w_{i\ftype\ell}&\log\left(\delta(\ell)/\calP_{i\ftype}\right)+\sum_{i=1}^n\sum_{\ell=1}^{|\calS_{i\mtype}|} w_{i\mtype\ell}\log\left(\delta(\ell)/\calP_{i\mtype}\right),\label{flow-obj}\\
\text{subject to}   \quad \sum_{\ell=1}^{|\calS_{i\ftype}|} w_{it\ell }&=z_{it} \quad\;\; \;\hspace{1pt}\text{ for every }i\in [n] \text{ and each }t\in \{\ftype,\mtype\},\label{flow1}\\
			z_{i\ftype}+z_{i\mtype}&=\calJ_i\quad\quad \text{ for every }i\in [n],\label{flow2}\\
\sum_{i\in [n]}z_{it}&=\phi_t\quad\quad\hspace{-2pt}\;\text{ for each }\;t\in \{\ftype,\mtype\}\label{flow3},\\
			\quad\quad 0\le \; w_{it\ell}&\le 1\quad\quad\;\;\;\hspace{-2pt}\text{ for every }i\in [n], t\in \{\ftype,\mtype\} \text{ and }\ell\in [\calS_{it}]\label{flow4}.	
\end{align}
The variable $z_{it}$ represents the total number of seats that are allocated in the solution for party $i\in [n]$ and type $t\in \{\ftype,\mtype\}$.
One could equivalently write the program above as a convex piecewise linear flow problem by not including the set of variables $w_{it\ell}$.
Constraint (\ref{flow2}) indicates that the allocation should respect the row marginals and constraint (\ref{flow3}) enforces the solution to satisfy the type marginals.
In the following let $(x,w)$ be an optimal solution of (\ref{flow-obj})-(\ref{flow4}) and let $\Lambda\in \RR^{I}$ and $\calU\in \RR^{J}$ be the dual solutions associated to the constraints (\ref{flow2}) and (\ref{flow3}) respectively.
We refer to the tuple $(x,w,\Lambda,\calU,\beta)$ as the optimal primal dual pair, which satisfies the following conditions,
\begin{align}
\Lambda_i+\calU_t+\beta_{it\ell}&\le \log(\delta(\ell)/\calP_{it}), \label{eq:dual1}\\
w_{it\ell}\left(\Lambda_i+\calU_t+\beta_{it\ell}-\log(\delta(\ell)/\calP_{it})\right)&=0,\label{eq:dual2}\\
\beta_{it\ell}(w_{it\ell}-1)&=0,\label{eq:dual3}\\
\beta_{it\ell}&\le 0,\label{eq:dual4}
\end{align}
for every party $i\in [n]$, each type $t\in \{\mtype,\ftype\}$ and every $\ell\in [\calS_{it}]$, where $\beta_{it\ell}$ is the dual variable associated to the upper bound in constraint (\ref{flow4}) on the value of $w_{it\ell}$.
The following result summarizes the main properties of this network flow problem.
We recall that by the network flow theory we know that every optimal extreme point $x$ of the problem (\ref{flow-obj})-(\ref{flow3}) is such that $x_{it}\in \ZZ$ for each party $i\in [n]$ and type $t\in \{\ftype,\mtype\}$.

\begin{lemma}
\label{lem:biprop-duality}
Let $(\calP,\calS,\calJ,\phi)$ be a two-dimensional instance with supply and let $x$ be a matrix with integer entries and dimensions $[n]\times \{\ftype,\mtype\}$.
Consider $\lambda\in \RR_+^n$ and $\mu=(\mu_{\ftype},\mu_{\mtype})\in \RR_+^{2}$.
Take $\Lambda_i=\log(\lambda_i)$ for every party $i\in [n]$ and consider $\calU_{\ftype}=\log(\mu_{\ftype})$ and $\calU_{\mtype}=\log(\mu_{\mtype})$.
Then, $(x,\lambda,\mu)$ is a biproportional solution for the two-dimensional instance $(\calP,\calS,\calJ,\phi)$ if and only if there exists an integer vector $w$ and a non-positive vector $\beta$ such that $(x,w,\Lambda,\calU,\beta)$ is an optimal primal dual pair of (\ref{flow-obj})-(\ref{flow3}).
\end{lemma}

\begin{proof}
Suppose that $(x,w,\Lambda,\calU,\beta)$ is an optimal primal dual pair of (\ref{flow-obj})-(\ref{flow3}) where $x$ is an extreme point.
We start by observing that for each $i\in [n]$ and each type $t\in \{\ftype,\mtype\}$, we have that $w_{it\ell}=1$ for $\ell\in \{1,\ldots,x_{it}\}$ and $w_{itk}=0$ for $k\in \{x_{it}+1,\ldots,\calS_{it}\}$.
This comes directly by the fact that for each $i\in [n]$ and each type $t\in \{\ftype,\mtype\}$ the function $\log(\delta(\ell)/\calP_{it})$ is strictly increasing as a function of $\ell$ and since $x_{it}\in \ZZ$.
Consider $i\in [n]$ and $t\in \{\ftype,\mtype\}$ such that $x_{it}<\calS_{it}$.
By condition (\ref{eq:dual2}) when $\ell=x_{it}$ we have that $\Lambda_i+\calU_t+\beta_{it\ell}-\log(\delta(x_{it})/\calP_{it})=0$, that is $\delta(x_{it})=\calP_{it}e^{\Lambda_i}e^{\calU_t}e^{\beta_{it\ell}}=\calP_{it}\lambda_i\mu_te^{\beta_{it\ell}}\le \calP_{it}\lambda_i\mu_t$,
where the last inequality comes from the fact that $\beta_{it\ell}\le 0$ by condition (\ref{eq:dual4}).
On the other hand, when $\ell=x_{it}+1$ we have that $x_{it\ell}=0$ and therefore the complementary slackness condition (\ref{eq:dual3}) implies that $\beta_{it\ell}=0$.
Then, condition (\ref{eq:dual1}) implies that $\Lambda_i+\calU_t\le \log(\delta(x_{it}+1)/\calP_{it})$, that is $\delta(x_{it}+1)\ge  \calP_{it}\lambda_i\mu_t$, and therefore $x_{it}\in \calR_{\delta}(\calP_{it}\lambda_i\mu_t)$.
We conclude that $(x,\lambda,\mu)$ is a biproportional solution with signpost sequence $\delta$.

Conversely, suppose that the triplet $(x,\lambda,\mu)$ is a biproportional solution with divisor $\delta$ for the two-dimensional instance $(\calP,\calS,\calJ,\phi)$. 
For each party $i\in [n]$ and each type $t\in \{\ftype,\mtype\}$, consider $w_{it\ell}=1$ for $\ell\le x_{it}$ and $w_{it\ell}=0$ is zero otherwise.
We remark this is possible since $x_{it}\le \calS_{it}$.
Furthermore, for each party $i\in [n]$ and each type $t\in \{\ftype,\mtype\}$, let $\beta_{it\ell}=\log(\delta(\ell)/\calP_{it})-\Lambda_i-\calU_t$ for $\ell\le x_{it}$ and $\beta_{it\ell}=0$ otherwise.
By construction, the tuple $(x,w,\Lambda,\calU,\beta)$ defined in this way satisfies the complementary slackness constraints (\ref{eq:dual2})-(\ref{eq:dual3}). 
Since $(x,\lambda,\mu)$ is a biproportional solution, we have that $\delta(x_{it})\le \calP_{it}\lambda_i\mu_t= \calP_{it}e^{\Lambda_i}e^{\calU_t}$, which implies that $\log(\delta(x_{it})/\calP_{it})-\Lambda_i-\calU_t\le 0$.
Therefore, by the monotonicity of the signpost sequence $\delta$ we have that 
\[\beta_{it\ell}= \log(\delta(\ell)/\calP_{it})-\Lambda_i-\calU_t\le \log(\delta(x_{it})/\calP_{it})-\Lambda_i-\calU_t\le 0\] 
for each $\ell\le x_{it}$, and in consequence constraint (\ref{eq:dual4}) is satisfied since $\beta_{it\ell}=0$ for $\ell>x_{it}$.
By construction of $\beta$ we have that constraint (\ref{eq:dual1}) is satisfied with equality when $\ell\le x_{it}$.
By the biproportionality of $(x,\lambda,\mu)$ we also have that $\delta(x_{it}+1)\ge \calP_{it}\lambda_i\mu_t= \calP_{it}e^{\Lambda_i}e^{\calU_t}$, and therefore $\log(\delta(x_{it}+1)/\calP_{it})\ge \Lambda_i+\calU_t$.
Thus, by the monotonicity of the signpost sequence $\delta$ we have that 
\[\log(\delta(\ell)/\calP_{it})\ge \log(\delta(x_{it}+1)/\calP_{it})\ge \Lambda_i+\calU_t=\Lambda_i+\calU_t+\beta_{it\ell}\] 
for each $\ell>x_{it}$, since $\beta_{it\ell}=0$ is zero in this case. 
By strong duality, we conclude that $(x,w,\Lambda,\calU,\beta)$ is an optimal primal dual pair.
\end{proof}

\subsection{Description of the Biproportional Parity Mechanism}

Given an instance $\calI$, we define $\calP(\calI)$ and $\calS(\calI)$ with entries in $[n]\times \{\ftype,\mtype\}$ such that $\calP_{it}(\calI)=\sum_{c\in \calC_{i}^t}\votes_{\calI}(c)$ and $\calS_{it}(\calI)=|\calC_{it}|$. 
We say that $t\in \{\ftype,\mtype\}$ is {\it vote leading} when the total number of votes garnered by candidates of type $t$ is strictly larger than the number of votes garnered by the other type.
In case of equality, the vote leading type is decided according to some tie breaking rule.
We assume that this tie breaking rule is invariant under scaling: If $t$ is vote leading in an instance $\calI$, then $t$ is vote leading in $\alpha \calI$ for every positive real $\alpha$.
The {\it parity marginal} for $\calI$, denoted by $\phi(\calI)$, is a vector with entries in $\{\ftype,\mtype\}$ such that $\phi_{\ftype}(\calI)+\phi_{\mtype}(\calI)=h$ and the following holds: When $h$ is even we have $\phi_{\ftype}(\calI)=\phi_{\mtype}(\calI)=h/2$, and when $h$ is odd we have $\phi_{t}(\calI)=\lceil h/2\rceil$ for the vote leading type $t\in \{\ftype,\mtype\}$.

We now describe the mechanism based on biproportionality.
We remark that in our approach the marginals of the biproportional problem are computed from the input.
In particular, the party marginals depend strongly on the votes, since they are computed by using an apportionment method at the level of the parties. 

\begin{algorithm}[H]
	\begin{algorithmic}[1]
		\Require{An instance $\calI$ satisfying the supply condition and $\calJ\in \apportiong(\iparty)$.}
		\Ensure{A set of allocations satisfying condition \ref{parity-a} and \ref{parity-b}.}
		\vspace{.2cm}
		\State For every $i\in [n]$ and each $t\in \{\ftype,\mtype\}$ define $\calP_{it}(\calI)=\sum_{c\in \calC_{i}^t}\votes_{\calI}(c)$ and $\calS_{it}(\calI)=|\calC_{it}|$.
		\State For every $x\in \calB_{\delta}(\calP(\calI),\calS(\calI),\calJ,\phi(\calI))$ do the following: For every $i\in [n]$ and each $t\in \{\ftype,\mtype\}$, define $\calE_{x}(c)=1$ for every $c\in \calC_i^t$ that belongs to the top $x_{it}$ candidates from $\calC_i^t,$ and zero otherwise.\label{step3}
		\State Return $\Theta_{\calJ}=\{\calE_x:x\in \calB_{\delta}(\calP(\calI),\calS(\calI),\calJ,\phi(\calI))\}$. 
	\end{algorithmic}
	\caption{\small Biproportional Parity \label{alg:biprop}}
\end{algorithm}
\noindent The {\it $\delta$-biproportional parity mechanism}, denoted by $\calM_{\delta}^{B}$, is defined as follows: For every pair $(\calI,\gamma)$ we have $\calM_{\delta}^{B}(\calI,\gamma)$ is equal to the union of the sets $\Theta_{\calJ}$ where $\calJ\in \apportiong(\iparty)$ and $\Theta_{\calJ}$ is the set computed by Algorithm \ref{alg:biprop}.
This algorithm computes a biproportional solution for an instance $(\calP(\calI),\calS(\calI),\calJ,\phi(\calI))$ where $\calJ\in \apportiong(\iparty)$ and $\phi(\calI)$ is the parity marginal for $\calI$.
The following is the first main result for this mechanism.

\begin{theorem}
	\label{thm:main}
	For every signpost sequence $\delta$ with $\delta(1)>0$ and every signpost sequence $\gamma$ we have that $\calM_{\delta}^{B}$ is $\gamma$-satisfactory over the instances satisfying the supply condition.
\end{theorem}

Before proving the theorem, we show that when the instance $\calI$ satisfies the supply condition, there exists a biproportional solution for the corresponding instance.

\begin{proposition}
\label{lem:biprop-supply-feasible}
Let $\calI$ be an instance satisfying the supply condition, let $\gamma$ be a signpost sequence and consider $\calJ\in \apportiong(\iparty)$.
Then, for every $\delta$ with $\delta(1)>0$ we have that $\calB_{\delta}(\calP(\calI),\calS(\calI),\calJ,\phi(\calI))\ne \emptyset$.
In particular, we have that $\calM_{\delta}^B(\calI,\gamma)\ne \emptyset$.
\end{proposition}

\begin{proof}
Thanks to Lemma \ref{lem:biprop-duality} it is sufficient to show that there exists a pair $(z,w)$ satisfying the constraints (\ref{flow1})-(\ref{flow4}) for the two-dimensional instance $(\calP(\calI),\calS(\calI),\calJ,\phi(\calI))$.	
By Lemma \ref{lem:closed-supply} and Proposition \ref{prop:linear-program} we have that there exists an integral vector $y\in \calZ(\calI,\calJ,\gamma)$.
It satisfies the following:
\begin{align*}
	\sum_{i\in [n]}y_{it}&\ge \lfloor h/2\rfloor \quad \text{ for each }t\in \{\ftype,\mtype\},\\
	y_{i\ftype}+y_{i\mtype}&=\calJ_i \quad\quad\;\; \text{ for every }i\in [n],\\
	0\le y_{it}&\le \calS_{it}(\calI)\quad \text{ for every }i\in [n]\text{ and each }t\in \{\ftype,\mtype\}.
\end{align*}
Given an integral $y$ satisfying the above set of inequalities we define the pair $(z(y),w(y))$ as follows: For every $i\in [n]$ and each $t\in \{\ftype,\mtype\}$ we have $w_{it\ell}(y)=1$ for each $\ell\in [y_{it}]$ and zero otherwise; $z_{it}(y)=\sum_{\ell=1}^{\calS_{it}}w_{it\ell}(y)$.
Suppose that $h$ is even.
Since $\sum_{i\in [n]}\calJ_i=h$, in this case we have that $\sum_{i\in [n]}y_{it}=h/2=\phi_t(\calI)$ for each $t\in \{\ftype,\mtype\}$ and by construction the pair $(z(y),w(y))$ satisfies (\ref{flow1})-(\ref{flow4}) for the two-dimensional instance $(\calP(\calI),\calS(\calI),\calJ,\phi(\calI))$.
Now suppose that $h$ is odd and let $\bar s$ be the vote leading type and $s$ the other type.
If $\sum_{i\in [n]}y_{i\bar s}=\lceil h/2\rceil$ then the pair $(z(y),w(y))$ satisfies (\ref{flow1})-(\ref{flow4}).
Otherwise, consider any $i\in [n]$ such that $y_{is}>0$, and define the solution $\bar y$ as follows: $\bar y_{i\bar s}=y_{i\bar s}+1$, $\bar y_{is}=y_{is}-1$ and $\bar y_{jt}= y_{jt}$ otherwise.
Observe that $y_{i\bar s}\le \lfloor h/2\rfloor$ and therefore $\bar y_{i\bar s}\le \lceil h/2\rceil\le \calS_{i\bar s}$.
Then, by construction we have that the pair $(z(\bar y),w(\bar y))$ satisfies (\ref{flow1})-(\ref{flow4}) for the two-dimensional instance $(\calP(\calI),\calS(\calI),\calJ,\phi(\calI))$.
\end{proof}

\begin{lemma}
\label{lem:biprop-correct}
Let $\calI$ be an instance satisfying the supply condition and consider $\calJ\in \apportiong(\iparty)$ for some signpost sequence $\gamma$.
Then, for every signpost sequence $\delta$, every solution in the output of Algorithm \ref{alg:biprop} satisfies conditions \ref{parity-a} and \ref{parity-b}.  
In particular, the $\delta$-biproportional parity mechanism $\calM_{\delta}^B$ is a valid apportionment mechanism over the instances satisfying the supply condition.
\end{lemma}

\begin{proof}
Let $\calI$ be an instance that satisfies the supply condition and let $\gamma$ be a signpost sequence.
Consider $\calJ\in \apportiong(\iparty)$ 
Thanks to Proposition \ref{lem:biprop-supply-feasible} we have that $\calM^B_{\delta}(\calI,\gamma)\ne \emptyset$ and therefore it is sufficient to show in what follows that $\calM^B_{\delta}(\calI,\gamma)\subseteq \calX(\calI,\gamma)$.
Let $\calE_{x}\in \Theta_{\calJ}$, where $\Theta_{\calJ}$ is the output of Algorithm \ref{alg:biprop} and $x\in (\calP(\calI),\calS(\calI),\calJ,\phi(\calI))$.
By construction it holds that $\sum_{c\in \calC_i}\calE_x(c)=x_{i\ftype}+x_{i\mtype}=\calJ_i$ for every $i\in [n]$ and therefore $\calE_x$ satisfies condition \ref{parity-b}.
On the other hand, we have that  $\sum_{c\in \calC^t}\calE_x(c)=\sum_{i\in [n]}x_{it}=\phi_t(\calI)$.
By definition of the parity marginals it holds that $|\phi_{\ftype}(\calI)-\phi_{\mtype}(\calI)|=(h\text{ mod } 2)$ and therefore $\calE_x$ satisfies condition\ref{parity-a}.
\end{proof}

Before proving Theorem~\ref{thm:main} we need a few more technical results about the biproportional solutions and the minimum cost flow problem (\ref{flow-obj})-(\ref{flow4}).
\begin{definition}
Given two different matrices $x$ and $\tilde x$ satisfying (\ref{biprop2})-(\ref{biprop3}), we say that two different parties $\pparty,\qparty\in [n]$ induce a {\it cycle} for $x$ and $\tilde x$ if the following inequalities are satisfied: $\tilde x_{\pparty\ftype}\ge x_{\pparty\ftype}+1$, $\tilde x_{\pparty\mtype}\le x_{\pparty\mtype}-1$, $\tilde x_{\qparty\mtype}\ge x_{\qparty\mtype}+1$ and $\tilde x_{\qparty\ftype}\le x_{\qparty\ftype}-1$.
\end{definition}
Observe that such two parties inducing a cycle are always guaranteed to exists when $x\ne \tilde x$ since both have the same row and column marginals.
The first statement in the following lemma corresponds to a structural result for the instances in which there is no unique biproportional solution, and that is closely related to these cycles.

\begin{lemma}
\label{lem:technical-biprop}
Let $(\calP,\calS,\calJ,\phi)$ be a two-dimensional instance with supply.
Then, for every signpost sequence $\delta$ the following holds:
\begin{enumerate}[label=(\alph*)]
	\item \label{opt:ties}
Let $(x,\lambda,\mu)$ and $(\tilde x,\tilde \lambda,\tilde \mu)$ be $\delta$-biproportional solutions for $(\calP,\calS,\calJ,\phi)$ with $x\ne \tilde x$. 
	Then, there exists two parties $\pparty,\qparty\in [n]$ inducing a cycle for $x$ and $\tilde x$ and such that 
	$\delta(x_{\pparty\ftype}+1)=\calP_{\pparty\ftype}\lambda_{\pparty}\mu_{\ftype}$,
	$\delta(x_{\pparty\mtype})=\calP_{\pparty\mtype}\lambda_{\pparty}\mu_{\mtype}$,
	$\delta(x_{\qparty\mtype}+1)=\calP_{\pparty\mtype}\lambda_{\pparty}\mu_{\mtype}$ and
	$\delta(x_{\qparty\ftype})=\calP_{\qparty\ftype}\lambda_{\qparty}\mu_{\ftype}$.
	\item For every positive real value $\alpha>0$, we have that $(x,\lambda,\mu)$ is a $\delta$-biproportional solution for $(\calP,\calS,\calJ,\phi)$ if and only if $(x,\frac{1}{\sqrt{\alpha}}\lambda,\frac{1}{\sqrt{\alpha}}\mu)$ is a biproportional solution for $(\alpha\calP,\calS,\calJ,\phi)$. \label{opt:scaling}
\end{enumerate}
\end{lemma}
To prove Lemma \ref{lem:technical-biprop} we use
the following result that holds even for more general instances where the number of columns is larger than two.
For simplicity, we state a form of the result that suffices for our purposes.

\begin{theorem}[\cite{Pukk2017}]
	\label{thm:cycles}
	Let $(\calG,\calJ,\phi)$ be a two-dimensional instance.
	Then, for every signpost sequence $\delta$, we have that $x\in \calB_{\delta}(\calG,\phi,\calJ)$ if and only if for any pair $i,j\in [n]$ we have that 
	\begin{align*}
		\frac{\delta(x_{i\ftype})}{\calG_{i\ftype}}\cdot \frac{\delta(x_{j\mtype})}{\calG_{j\mtype}}\le \frac{\delta(x_{i\mtype}+1)}{\calG_{i\mtype}}\cdot \frac{\delta(x_{j\ftype}+1)}{\calG_{j\ftype}}\quad \text{ and }\quad
		\frac{\delta(x_{i\mtype})}{\calG_{i\mtype}}\cdot \frac{\delta(x_{j\ftype})}{\calG_{j\ftype}}\le \frac{\delta(x_{i\ftype}+1)}{\calG_{i\ftype}}\cdot \frac{\delta(x_{j\mtype}+1)}{\calG_{j\mtype}}.
	\end{align*}
	Furthermore, $x$ is the unique $\delta$-biproportional solution if every inequality is satisfied strictly.
\end{theorem}

\begin{proof}[Proof of Lemma \ref{lem:technical-biprop}]
Since $x\ne \tilde x$ and $\sum_{i\in [n]}x_{i\ftype}=\sum_{i\in [n]}\tilde x_{i\ftype}$, there exist two parties $\pparty,\qparty\in [n]$ such that $\tilde x_{\pparty\ftype}\ge x_{\pparty\ftype}+1$ and $\tilde x_{\qparty\ftype}\le x_{\qparty\ftype}-1<\calS_{\qparty,\ftype}$.
It follows that $\pparty$ and $\qparty$ induce a cycle for $x$ and $\tilde x$ since $x_{\pparty\ftype}+x_{\pparty\mtype}=\tilde x_{\pparty\ftype}+\tilde x_{\pparty\mtype}$ and $x_{\qparty\ftype}+x_{\qparty\mtype}=\tilde x_{\qparty\ftype}+\tilde x_{\qparty\mtype}$ and therefore $\tilde x_{\pparty\mtype}\le x_{\pparty\mtype}-1<\calS_{\pparty,\mtype}$ and $\tilde x_{\qparty\mtype}\ge x_{\qparty\mtype}+1$.
By Theorem~\ref{thm:cycles} we have that
\begin{align*}
\frac{\delta(\tilde x_{\pparty\ftype})}{\calP_{\pparty\ftype}}\cdot \frac{\delta(\tilde x_{\qparty\mtype})}{\calP_{\qparty\mtype}}\le 
 \frac{\delta(\tilde x_{\pparty\mtype}+1)}{\calP_{\pparty\mtype}}\cdot \frac{\delta(\tilde x_{\qparty\ftype}+1)}{\calP_{\qparty\ftype}},\quad 
\frac{\delta(x_{\pparty\mtype})}{\calP_{\pparty\mtype}}\cdot \frac{\delta(x_{\qparty\ftype})}{\calP_{\qparty\ftype}}\le 
 \frac{\delta(x_{\pparty\ftype}+1)}{\calP_{\pparty\ftype}}\cdot \frac{\delta(x_{\qparty\mtype}+1)}{\calP_{\qparty\mtype}}.
\end{align*}
The above inequalities, together with the fact that $\pparty$ and $\qparty$ induces a cycle for $x$ and $\tilde x$, and the monotnicity of the signpost sequence $\delta$, imply that 
\begin{align*}
\frac{\delta(x_{\pparty\mtype})}{\calP_{\pparty\mtype}}\cdot \frac{\delta(x_{\qparty\ftype})}{\calP_{\qparty\ftype}}&\le 
\frac{\delta(x_{\pparty\ftype}+1)}{\calP_{\pparty\ftype}}\cdot \frac{\delta(x_{\qparty\mtype}+1)}{\calP_{\qparty\mtype}}\\
 &\le\frac{\delta(\tilde x_{\pparty\ftype})}{\calP_{\pparty\ftype}}\cdot \frac{\delta(\tilde x_{\qparty\mtype})}{\calP_{\qparty\mtype}}\le 
\frac{\delta(\tilde x_{\pparty\mtype}+1)}{\calP_{\pparty\mtype}}\cdot \frac{\delta(\tilde x_{\qparty\ftype}+1)}{\calP_{\qparty\ftype}}\le 
\frac{\delta(x_{\pparty\mtype})}{\calP_{\pparty\mtype}}\cdot \frac{\delta(x_{\qparty\ftype})}{\calP_{\qparty\ftype}},
\end{align*}
and therefore the chain of inequalities is actually a chain of equalities. 
Then, we recover on one side that $\delta(x_{\pparty\mtype})=\calP_{\pparty\mtype}\lambda_{\pparty}\mu_{\mtype}$ and $\delta(x_{\qparty\ftype})=\calP_{\qparty\ftype}\lambda_{\qparty}\mu_{\ftype}$, and from the other side we recover $\delta(x_{\pparty\ftype}+1)=\calP_{\pparty\ftype}\lambda_{\pparty}\mu_{\ftype}$ and $\delta(x_{\qparty\mtype}+1)=\calP_{\qparty\mtype}\lambda_{\qparty}\mu_{\mtype}$.
We have that \ref{opt:scaling} follows since the scaling performed on the multipliers leaves invariant the quantities in (\ref{biprop1}).
\end{proof}

\begin{proof}[Proof of Theorem~\ref{thm:main}]
By Lemma \ref{lem:biprop-correct} we have that the $\delta$-biproportional parity mechanism is a valid apportionment mechanism over the instances satisfying the supply condition.
Therefore, it remains to check that mechanism satisfies the properties \ref{axiom2}-\ref{axiom3}-\ref{axiom4}.
Recall that for every instance $\calI$ we denote by $\calP(\calI)$ the matrix with entries in $[n]\times \{\ftype,\mtype\}$ such that $\calP_{it}(\calI)=\sum_{c\in \calC_{i}^t}\votes_{\calI}(c)$.\\

\noindent{\bf Exactness.} Now suppose that we are given an instance where the function $\votes_{\calI}$ is such that $\votes_{\calI}(c)\in \{0,1\}$ for every candidate $c\in \calC$ and the function $\votes_{\calI}$ satisfies \ref{parity-a} and \ref{parity-b}.
That is, there are exactly $h$ votes and exactly $h$ candidates obtain exactly one vote each. 
Since $\votes_{\calI}$ satisfies \ref{parity-a}, we have that $\sum_{c\in \calC}\votes_{\calI}(c)=h=\sum_{i\in [n]}(\calP_{i\ftype}(\calI)+\calP_{i\mtype}(\calI))$ and
therefore, we have that 
$\calQ_i(\calI)=\calP_{i\ftype}(\calI)+\calP_{i\mtype}(\calI)$ for each party $i\in [n]$ and $\sum_{i\in [n]}\calQ_i(\calI)=h$,
which by Lemma~\ref{lem:jefferson-one-dim}~\ref{jeff-proportional} with $\lambda=1$ implies that for every signpost sequence $\gamma$ we have that $\apportiong(\iparty)=\{\calJ\}$ where $\calJ_i=\calP_{i\ftype}(\calI)+\calP_{i\mtype}(\calI)$ for each party $i\in [n]$.
Therefore, for every signpost sequence $\gamma$ we have a unique instance $(\calP(\calI),\calS(\calI),\calJ,\phi(\calI))$ in this case.

For each party $i\in [n]$ consider the all ones vector $\mathrm{I}^n_i=1$ and for the type consider $\mathrm{I}_{\ftype}=\mathrm{I}_{\mtype}=1$.
We claim that the triplet $(\calP(\calI),\mathrm{I}^n,\mathrm{I})$ is the unique $\delta$-biproportional solution $(\calP(\calI),\calS(\calI),\calJ,\phi(\calI))$.
We start by showing that it is a biproportional solution.
The conditions (\ref{biprop1}) and (\ref{biprop4}) are clearly satisfied since $\calP_{it}(\calI)\in \calR_{\delta}(\calP_{it}(\calI))$ for each $i\in [n]$ and each $t\in \{\ftype,\mtype\}$, and since $\sum_{c\in \calC_i^t}\votes_{\calI}(c)=\calP_{it}(\calI)\le \calS_{it}(\calI)$.
We now check that $x=\calP(\calI)$ satisfies (\ref{biprop2})-(\ref{biprop3}).
Since the function $\votes_{\calI}$ is binary and it satisfies \ref{parity-a}, we have that $\sum_{i\in [n]}\calP_{it}(\calI)=\sum_{c\in \calC^t}\votes_{\calI}(c)=\phi_t(\calI)$ for each type $t\in \{\ftype,\mtype\}$, and therefore $(\calP(\calI),\mathrm{I}^n,\mathrm{I})$ satisfies (\ref{biprop3}).
Since $\votes_{\calI}$ satisfies \ref{parity-b}, we have that $\sum_{c\in \calC_i}\votes_{\calI}(c)=\calJ_i=\calP_{i\ftype}+\calP_{i\mtype}$ for every $i\in [n]$.
Then, $(\calP(\calI),\mathrm{I}^n,\mathrm{I})$ satisfies constraint (\ref{biprop2}).
 
Suppose there exists a different solution $(\tilde x,\tilde \lambda,\tilde \mu)$ for the instance.
By Lemma~\ref{lem:technical-biprop}~\ref{opt:ties}, we have that there exist two parties
$\pparty,\qparty\in [n]$ inducing a cycle for $\calP(\calI)$ and $\tilde x$ and therefore 
$\delta(\calP_{\pparty\ftype}(\calI)+1)=\calP_{\pparty\ftype}(\calI)$ and $\delta(\calP_{\pparty\mtype}(\calI))=\calP_{\pparty\mtype}(\calI)$, but this goes in contradiction to the disjunction property \ref{divisor3} satisfied by the signpost sequence $\delta$.
Therefore, $(\calP(\calI),\mathrm{I}^n,\mathrm{I})$ is the unique $\delta$-biproportional solution in this case.
In consequence, step~\ref{step3} of Algorithm \ref{alg:biprop} allocates exactly one seat to each of the candidates that got exactly one vote and it outputs $\{\votes_{I}\}$.
This shows that property \ref{axiom2} is satisfied. \\

\noindent{\bf Scaling.} Let $\alpha$ be any positive real and consider the instance $\alpha\calI=(\calC,\party_{\calI},\alpha\cdot \votes_{\calI},\type_{\calI},h)$ obtained by scaling the number of votes. 
By Lemma \ref{lem:jefferson-one-dim} \ref{jeff-scaling} we have $\apportiong(\calQ(\calI))=\apportiong(\calQ(\alpha\calI))$.
Furthermore, by Lemma~\ref{lem:technical-biprop}~\ref{opt:scaling}, for every signpost sequence $\gamma$ and every $\calJ\in \apportiong(\iparty)=\apportiong(\calQ(\alpha\calI))$ the $\delta$-biproportional solutions of the scaled instance $(\alpha \calP(\calI),\calS(\calI),\calJ,\phi(\calI))$ coincides with the set of solutions for $(\calP(\calI),\calS(\calI),\calJ,\phi(\calI))$.
Therefore, Algorithm \ref{alg:biprop} computes the same allocations as in the instance $\calI$.
This shows that property \ref{axiom3} is satisfied.\\

\noindent{\bf Monotonicity.} Let $\calI$ and $\calG$ be two instances satisfying the supply condition such that $\calG$ is a voting increment of $\calI$ in $c\in \calC$.
Suppose that $c\in \calC_p$ and let $\gtype$ be the type of $c$.
In particular, we have that $\calP_{\pparty \gtype}(\calI)< \calP_{\pparty \gtype}(\calG)$ and $\calP_{it}(\calI)= \calP_{it}(\calG)$ for every pair $(i,t)\ne (\pparty,\gtype)$.
Consider $\calJ\in \apportiong(\iparty)$.
By Lemma~\ref{lem:jefferson-one-dim}~\ref{jeff-monotone-strong} there exists $\calJ'\in \apportiong(\calQ(\calG))$ such that $\calJ_{\pparty}\le \calJ_{\pparty}'$ and $\calJ_{i}\ge \calJ_{i}'$ for every $i\ne \pparty$.
In order to show the monotonicity property it is sufficient to prove that for every biproportional solution $(x(\calI),\lambda(\calI),\mu(\calI))$ of $(\calP(\calI),\calS(\calI),\calJ,\phi(\calI))$ and every biproportional solution $(x(\calG),\lambda(\calG),\mu(\calG))$ of $(\calP(\calG),\calS(\calG),\calJ',\phi(\calG))$ we have that
\begin{equation}
x_{\pparty \gtype}(\calI)\le x_{\pparty \gtype}(\calG).\label{goal}
\end{equation}
From this the property follows, since the position of the candidate $c$ in the total order $(\calC,\succ_{\calG})$ is at least the position in the total order $(\calC,\succ_{\calI})$ and therefore $\calE_{x(\calG)}(c)\ge \calE_{x(\calI)}(c)$.
Observe that by Lemma \ref{lem:jefferson-one-dim} \ref{jeff-monotone} the inequality (\ref{goal}) holds immediately if there is a unique party.
Otherwise, suppose the claim does not hold, that is, suppose that there exists two solutions $x(\calI)$ and $x(\calG)$ such that $x_{\pparty \gtype}(\calI)> x_{\pparty \gtype}(\calG)$ and $n\ge 2$. 
Assume that $\gtype=\ftype$.
\begin{claim}
\label{claim:cycle}
If $x_{\pparty \ftype}(\calI)> x_{\pparty \ftype}(\calG)$, there exists a party $\mathsf{q}\ne \pparty$ such that the following holds: $(a)\; x_{\pparty \ftype}(\calG)\le x_{\pparty \ftype}(\calI)-1,\; (b)\; x_{\pparty \mtype}(\calG)\ge x_{\pparty \mtype}(\calI)+1,
(c)\; x_{\qparty \ftype}(\calG)\ge x_{\qparty \ftype}(\calI)+1, and\; (d)\; x_{\qparty \mtype}(\calG)\le x_{\qparty \mtype}(\calI)-1$.		  
\end{claim}
\noindent We prove the claim after finishing the proof of the lemma.
The claim, together with the fact that $\calP_{\pparty \ftype}(\calI)< \calP_{\pparty \ftype}(\calG)$ and $\calP_{\qparty \mtype}(\calI)= \calP_{\qparty \mtype}(\calG)$, implies that 
\begin{equation}
\frac{\delta(x_{\pparty\ftype}(\calI))}{\calP_{\pparty\ftype}(\calI)}\cdot \frac{\delta(x_{\qparty\mtype}(\calI))}{\calP_{\qparty\mtype}(\calI)}
> \frac{\delta(x_{\pparty\ftype}(\calG)+1)}{\calP_{\pparty\ftype}(\calG)}\cdot \frac{\delta(x_{\qparty\mtype}(\calG)+1)}{\calP_{\qparty\mtype}(\calG)}\label{monotone1}
\end{equation}
By Theorem~\ref{thm:cycles}, together with parts $(c)$ and $(b)$ of the claim, and the fact that $\calP_{\qparty \ftype}(\calI)= \calP_{\qparty \ftype}(\calG)$ and $\calP_{\pparty \mtype}(\calI)= \calP_{\pparty \mtype}(\calG)$, we can lower bound the right hand side of the above inequality as follows,
\begin{align*}
\frac{\delta(x_{\pparty\ftype}(\calG)+1)}{\calP_{\pparty\ftype}(\calG)}\cdot \frac{\delta(x_{\qparty\mtype}(\calG)+1)}{\calP_{\qparty\mtype}(\calG)}&\ge \frac{\delta(x_{\qparty\ftype}(\calG))}{\calP_{\qparty\ftype}(\calG)}\cdot \frac{\delta(x_{\pparty\mtype}(\calG))}{\calP_{\pparty\mtype}(\calG)}\\
&\ge \frac{\delta(x_{\qparty\ftype}(\calI)+1)}{\calP_{\qparty\ftype}(\calI)}\cdot \frac{\delta(x_{\pparty\mtype}(\calI)+1)}{\calP_{\pparty\mtype}(\calI)}\ge \frac{\delta(x_{\pparty\ftype}(\calI))}{\calP_{\pparty\ftype}(\calI)}\cdot \frac{\delta(x_{\qparty\mtype}(\calI))}{\calP_{\qparty\mtype}(\calI)},
\end{align*}
which contradicts inequality (\ref{monotone1}).
We recall that the first and third inequality in the above chain follow by applying Theorem~\ref{thm:cycles}.
Therefore property \ref{axiom4} is satisfied.
We now prove Claim \ref{claim:cycle}.
Since $\calJ'_{\pparty}\ge \calJ_{\pparty}$ we have  $x_{\pparty\ftype}(\calG)+x_{\pparty\mtype}(\calG)\ge x_{\pparty\ftype}(\calI)+x_{\pparty\mtype}(\calI)$.
Since $x_{\pparty \ftype}(\calG)\le x_{\pparty \ftype}(\calI)-1$, we conclude that $x_{\pparty \mtype}(\calG)-1\ge x_{\pparty \mtype}(\calI)$, that proves $(b)$.
Since the number of seats to allocate in both $\calI$ and $\calG$ is the same, we have that $\phi_{\ftype}(\calG)\ge \phi_{\ftype}(\calI)$.
Therefore, we have
\[x_{\pparty\ftype}(\calG)+\sum_{i\ne \pparty}x_{i\ftype}(\calG)=\phi_{\ftype}(\calG)\ge \phi_{\ftype}(\calI)\ge x_{\pparty\ftype}(\calI)+\sum_{i\ne \pparty}x_{i\ftype}(\calI)>x_{\pparty\ftype}(\calG)+\sum_{i\ne \pparty}x_{i\ftype}(\calI),\]
which implies the existence of a party $\qparty\ne \pparty$ such that $x_{\qparty\ftype}(\calG)>x_{\qparty\ftype}(\calI)$.
On the other hand, we have that $\calJ'_{\qparty}\le \calJ_{\qparty}$ and 
therefore $x_{\qparty\ftype}(\calG)+x_{\qparty\mtype}(\calG)\le x_{\qparty\ftype}(\calI)+x_{\qparty\mtype}(\calI)$.
Since $x_{\qparty\ftype}(\calG)>x_{\qparty\ftype}(\calI)$ we conclude that $x_{\qparty\mtype}(\calG)< x_{\qparty\mtype}(\calI)$. 
The integrality of both $x(\calI)$ and $x(\calG )$ implies $(c)$ and $(d)$.
\end{proof}

\section{The Fair Share as Benchmark}
\label{sec:fairness}

Recall that the goal in the biproportional setting is to achieve proportionality in both dimensions, while at the same time finding an integral solution.
If one relaxes the integrality condition, then such a solution is known as {\it fair share} or {\it matrix scaling}, an object that has been studied extensively in the optimization, statistics and algorithms communities.
This solution is also used as benchmark in order to evaluate the proportionality obtained by two-dimensional apportionment methods \cite{maier2010divisor}.
In what follows we restrict attention to the case in which the entries of the supply matrix $\calS$ of a two-dimensional instance $(\calP,\calS,\calJ,\phi)$ are all equal to $\phi_{\ftype}+\phi_{\mtype}$.
In this case, condition (\ref{biprop4}) becomes redundant and we just refer to $(\calP,\calJ,\phi)$ as a two-dimensional instance.

\begin{definition}
	\label{def:fairshare}
	Let $(\calP,\calJ,\phi)$ be a two-dimensional instance.
	We say that a strictly positive matrix $\calF$ of dimensions $[n]\times \{\ftype,\mtype\}$ is a {\it fair share} of $(\calP,\calJ,\phi)$ if there exists a strictly positive real vector $\lambda\in \RR_+^{n}$ and $\mu_{\ftype},\mu_{\mtype}\in \RR_+$ such that for each $i\in [n]$ and each $t\in \{\ftype,\mtype\}$ the following holds:
	\begin{align}
		\calF_{it}&= \calP_{it}\lambda_i \mu_t,\label{fs1}\\
		\calF_{i\ftype}+\calF_{i\mtype}&=\calJ_i,\label{fs2}\\
		\sum_{i\in [n]}\calF_{it}&=\phi_t.\label{fs3}
	\end{align}
	We say that $(\calF,\lambda,\mu)$ is the fair share tuple for the two-dimensional instance $(\calP,\calJ,\phi)$.
\end{definition}
Conditions (\ref{fs2}) and (\ref{fs3}) guarantee that the fair share satisfies the marginals, while condition (\ref{fs1}) ensures the proportionality of the solutions according to both dimensions.
We start by describing the convex optimization program that determines the value of the fair share.
Let $(\calP,\calJ,\phi)$ be a two-dimensional instance with marginals $\calJ$ and $\phi$ and strictly positive $\calP$.
Consider the following strictly convex optimization problem,
\begin{align}
\text{minimize} \quad\sum_{i\in [n]} y_{i\ftype}\Big(\log&\left(y_{i\ftype}/\calP_{i\ftype}\right)-1\Big)+\sum_{i\in [n]}y_{i\mtype}\Big(\log\left(y_{i\mtype}/\calP_{i\mtype}\right)-1\Big)\label{cvx-flow-obj}\\
\text{subject to} \quad\quad  y_{i\ftype}+y_{i\mtype}&=\calJ_i\quad \quad\hspace{-3pt}\text{ for every }i\in [n],\label{cvx-flow1}\\
\sum_{i\in [n]}y_{it}&=\phi_t\quad\;\;\; \text{ for each }t\in \{\ftype,\mtype\}\label{cvx-flow2},\\
			\quad\quad y_{it}&\ge 0 \quad\quad\;\text{ for every }i\in [n] \text{ and each }t\in \{\ftype,\mtype\}.\label{cvx-flow3}	
\end{align}
Constraints (\ref{cvx-flow1}) and (\ref{cvx-flow2}) enforces every solution to satisfy the party and type marginals respectively.
Given a vector $\omega$, we denote by $\exp(\omega)$ the vector obtained by applying the exponential to each of the entries of $\omega$.
\begin{proposition}
	\label{lem:fair-share-duality}
	Let $(\calP,\calJ,\phi)$ be a two-dimensional instance with $\calP$ strictly positive.
	Then, there exists a unique optimal primal dual solution of (\ref{cvx-flow-obj})-(\ref{cvx-flow3}).
	Furthermore, $(\calF,\exp(\Lambda),\exp(\calU))$ is a fair share of $(\calP,\calJ,\phi)$ if and only if $(\calF,\Lambda,\calU)$ is an optimal primal dual solution of (\ref{cvx-flow-obj})-(\ref{cvx-flow3}).
\end{proposition}
\begin{proof}
	The optimization problem  (\ref{cvx-flow-obj})-(\ref{cvx-flow3}) is strictly convex and satisfies the Slater condition.
	Then, there exists a unique solution $\calF$ and it is strictly positive.	
	For a optimal primal dual pair $(\calF,\lambda,\mu)$, the KKT optimality conditions are equivalent to $\log(\calF_{it}/\calP_{it})-\Lambda_i-\calU_t=0$ for every $i\in [n]$ and $t\in \{\ftype,\mtype\}$, that is, $\calF_{it}=\calP_{it}\cdot \exp(\Lambda_i)\cdot \exp(\calU_t).$
\end{proof}	

\begin{proposition}
	\label{prop:already-fair}
	Let $(\calG,\calJ,\phi)$ be a two-dimensional instance with $\calG$ strictly positive and such that the following holds:
	$\sum_{i\in [n]}\calG_{it}=\phi_t$ for each $t\in \{\ftype,\mtype\}$ and
	$\calG_{i\ftype}+\calG_{i\mtype}=\calJ_i$ for each $i\in [n]$.
	Then, $\calG$ is the fair share of $(\alpha\calG,\calJ,\phi)$ for every positive real $\alpha$.
\end{proposition}

\begin{proof}
	For every $i\in [n]$ define $\lambda_i=1/\sqrt{\alpha}$ and for each $t\in \{\ftype,\mtype\}$ define $\mu_t=1/\sqrt{\alpha}$.
	For every $i\in [n]$ and for each $t\in \{\ftype,\mtype\}$ we have that $\calG_{it}=(\alpha\calG_{it})\lambda_i\mu_t$, and therefore $\calG$ satisfies conditions (\ref{fs1})-(\ref{fs3}) defining the fair share of $\calG$.
	The uniqueness of the fair share implies that $\calG$ is the fair share of $(\alpha\calG,\calJ,\phi)$.
\end{proof}

\subsection{Positive Result for two-dimensional Instances with Two Rows and Two Columns}
\label{sec:two-rows}

Consider the particular case of two-dimensional instances of $2\times 2$.
In the following result we show that for this case the biproportional solution does not violate the bounds given by rounding (up or down) the fair share.

\begin{theorem}
	\label{thm:two-rows}
	For every two-dimensional instance $(\calP,\calJ,\phi)$ where $\calP$ is a strictly positive matrix of $2\times 2$, and for every signpost sequence $\delta$, we have $\lfloor \calF_{it} \rfloor \le x_{it}\le \lceil \calF_{it} \rceil$ for every $x\in \calB_{\delta}(\calP,\calJ,\phi)$, where $\calF$ is the fair share of $(\calP,\calJ,\phi)$.
\end{theorem}

\begin{proof}[Proof of Theorem~\ref{thm:two-rows}]
Let $x\in \calB_{\delta}(\calP,\calJ,\phi)$ and suppose that $x_{1\ftype}\ge \lceil\calF_{1\ftype}\rceil+1$.
Since both $x$ and $\calF$ have the same row and column marginals, 
we have that $x_{2\ftype}\le \lfloor \calF_{2\ftype}\rfloor-1$, $x_{1\mtype}\le \lfloor\calF_{1\mtype}\rfloor-1$ and $x_{2\mtype}\ge \lceil \calF_{2\mtype}\rceil+1$.
In particular, $\calF_{1\mtype}\ge 1$ and $\calF_{2\ftype}\ge 1$.
Let $(\calF,\lambda,\mu)$ be the fair share tuple and consider $\Phi=\lambda_1\lambda_2\mu_{\ftype}\mu_{\mtype}$.
Therefore, and since $\delta$ is a signpost sequence, we have that 
\begin{align*}
\frac{\delta(x_{1\ftype})}{\calP_{1\ftype}}\cdot \frac{\delta(x_{2\mtype})}{\calP_{2\mtype}}=\Phi\cdot \frac{\delta(x_{1\ftype})}{\calF_{1\ftype}}\cdot \frac{\delta(x_{2\mtype})}{\calF_{2\mtype}}&\ge \Phi\cdot \frac{\delta(\lceil\calF_{1\ftype}\rceil+1)}{\calF_{1\ftype}}\cdot \frac{\delta(\lceil \calF_{2\mtype}\rceil+1)}{\calF_{2\mtype}}\ge \Phi\cdot \frac{\lceil\calF_{1\ftype}\rceil}{\calF_{1\ftype}}\cdot \frac{\lceil \calF_{2\mtype}\rceil}{\calF_{2\mtype}}\ge \Phi.
\end{align*}
On the other hand, we have that 
\begin{align*}
\frac{\delta(x_{1\mtype}+1)}{\calP_{1\mtype}}\cdot \frac{\delta(x_{2\ftype}+1)}{\calP_{2\ftype}}&=\Phi\cdot \frac{\delta(x_{1\mtype}+1)}{\calF_{1\mtype}}\cdot \frac{\delta(x_{2\ftype}+1)}{\calF_{2\ftype}}\\
&\le \Phi\cdot \frac{\delta(\lfloor\calF_{1\mtype}\rfloor)}{\calF_{1\mtype}}\cdot \frac{\delta(\lfloor\calF_{2\ftype}\rfloor)}{\calF_{2\ftype}}\le \frac{\lfloor\calF_{1\mtype}\rfloor}{\calF_{1\mtype}}\cdot \frac{\lfloor\calF_{2\ftype}\rfloor}{\calF_{2\ftype}}
\le \Phi.
\end{align*}
If any of the entries of the fair share $\calF$ is fractional, then at least one of the inequalities above is strict, and threfore this contradicts the first set of inequalities in Theorem~\ref{thm:cycles}.
Otherwise, suppose that $\calF$ is integral.
In this case, and by Theorem~\ref{thm:cycles}, all the above inequalities are satisfied with equality, from where we get that 
$\delta(\lceil \calF_{1\ftype}\rceil+1)=\lceil \calF_{1\ftype}\rceil$ and $\delta(\lfloor \calF_{1\mtype}\rfloor-1)=\lfloor \calF_{1\mtype}\rfloor$, but this contradicts the disjunction property \ref{divisor3} satisfied by the signpost sequence $\delta$.
We conclude that $x_{1\ftype}\le \lceil \calF_{1\ftype}\rceil$.
By an analogous reasoning we show that when $x_{1\ftype}\le \lceil\calF_{1\ftype}\rceil-1$ the second set of inequalities in Theorem~\ref{thm:cycles} is contradicted.
Therefore, we have that $\lfloor \calF_{1\ftype}\rfloor\le x_{1\ftype}\le \lceil \calF_{1\ftype}\rceil$.
Since both $x$ and $\calF$ have the same marginal for $\ftype$, we conclude that $\lfloor \calF_{2\ftype}\rfloor\le x_{2\ftype}\le \lceil \calF_{2\ftype}\rceil$, and since both $x$ and $\calF$ have the same row marginals we conclude that $\lfloor \calF_{i\mtype}\rfloor\le x_{i\mtype}\le \lceil \calF_{i\mtype}\rceil$ for each $i\in \{1,2\}$.
\end{proof}

\subsection{Negative Results for two-dimensional Instances}
\label{sec:biprop-no-breach}
We show in our next result that the distance between the biproportional solution and the fair share can be arbitrarily high, in the following precise sense.

\begin{theorem}
	\label{thm:how-much-big}
	For every positive integer $\ell$ and every signpost sequence $\delta$, there exists a two-dimensional instance $\calD_{\ell,\delta}$ such that the following holds:
	There exists a unique $\delta$-biproportional solution $y$ of $\calD_{\ell,\delta}$, such that $y_{1\ftype}\ge \calF_{1t}+\ell$ and $y_{1\mtype}\le \calF_{1\mtype}-\ell$, where $\calF$ is the fair share of $\calD_{\ell,\delta}$.
	In particular, $\|y-\calF\|_1=\Omega(\ell)$.
\end{theorem} 
That is, for every integer value $\ell$ and every signpost sequence $\delta$, we can find an instance where the allocation for the first row in the biproportional solution differs by $\ell$ from the fair share in each of its entries.
To prove Theorem~\ref{thm:how-much-big} we need the following simple proposition.
\begin{proposition}
	\label{prop:existence1}
	Let $\delta$ be a signpost sequence such that $\delta(1)>0$ and let $\ell$ be a positive integer. 
	Consider the function $\Gamma_{\ell,\delta}:\RR_+\to \RR$ given by $\Gamma_{\ell,\delta}(y)=\delta(7+\ell)\cdot \delta(3)/(21-7y)-\delta(1)^2/(\ell y).$
	Then, there exists a positive integer number $n_{\ell,\delta}$ such that $\Gamma_{\ell,\delta}(\ell/n_{\ell,\delta})<0$. 
\end{proposition}

\begin{proof}
	The function $\Gamma_{\ell,\delta}$ is continuous in the interval $(0,1)$ and since $\delta(1)>0$ we have $\Gamma_{\ell,\delta}(y)\to -\infty$ when $y\to 0$.
	Therefore, there exists a large enough positive integer number $n_{\ell,\delta}$ such that $\Gamma_{\ell,\delta}(\ell/n_{\ell,\delta})<0$.
\end{proof}

\begin{proof}[Proof of Theorem~\ref{thm:how-much-big}]
	Let $\delta$ be a signpost sequence such that $\delta(1)>0$ and let $\ell$ be a positive integer.
	Let $n_{\ell,\delta}$ be the integer number guaranteed to exist by Proposition~\ref{prop:existence1} and let $y_{\ell,\delta}=\ell/n_{\ell,\delta}$.
	Consider the matrix $\calP$ with $n_{\ell,\delta}+1$ rows and two columns defined as follows: $\calP_{1\ftype}=7$ and $\calP_{1\mtype}=\ell$; $\calP_{i\ftype}=y_{\ell}$ and $\calP_{i\mtype}=3-y_{\ell}$ for every $i\in \{2,3,\ldots,n_{\ell,\delta}+1\}$.
	We define the row marginals $\calJ$ such that $\calJ_1=7+\ell$ and $\calJ_i=3$ for every $i\in \{2,3,\ldots,n_{\ell,\delta}+1\}$, and the type marginals $\phi$ such that $\phi_{\ftype}=7+n_{\ell,\delta}y_{\ell,\delta}=7+\ell$ and $\phi_{\mtype}=\ell+n_{\ell,\delta}(3-y_{\ell,\delta})=3n_{\ell,\delta}$.
	
	By construction the type marginals are integral.
	Consider the instance $\calD_{\ell,\delta}=(n_{\ell,\delta} \calP,\calJ,\phi)$.
	Since the matrix $\calP$ is such that $\sum_{i\in [n_{\ell,\delta}+1]}\calP_{it}=\phi_t$ for each $t\in \{\ftype,\mtype\}$ and $\calP_{i\ftype}+\calP_{i\mtype}=\calJ_i$ for each $i\in [n_{\ell,\delta}+1]$, by Proposition~\ref{prop:already-fair} we have that $\calP$ is the fair share of $(n_{\ell,\delta}\calP,\calJ,\phi)$.
	Furthermore, by Lemma \ref{lem:technical-biprop} \ref{opt:scaling} we have that $\calB_{\delta}(n_{\ell,\delta} \calP,\calJ,\phi)=\calB_{\delta}(\calP,\calJ,\phi)$.
	Therefore, it is sufficient to compare the biproportional solutions of $(\calP,\calJ,\phi)$ with respect to $\calP$.
	\begin{figure}[H]
		$$\calP=\begin{pmatrix} 7 & \ell\\ y_{\ell,\delta} & 3-y_{\ell,\delta}\\ \vdots & \vdots \\y_{\ell,\delta} & 3-y_{\ell,\delta} \end{pmatrix}\quad \quad \calJ=\begin{pmatrix} 7+\ell \\ 3\\ \vdots\\ 3\end{pmatrix}\quad \quad \phi=\begin{pmatrix}7+\ell\\ 3n_{\ell,\delta} \end{pmatrix}\quad \quad x=\begin{pmatrix} 7+\ell & 0\\ 0 & 3\\ \vdots & \vdots \\0 & 3 \end{pmatrix}$$
		\caption{Two-dimensional instance and its biproportional solution when $\delta(1)>0$.}
	\end{figure}
	\noindent Consider the matrix $x$ defined as follows: $x_{1\ftype}=7+\ell$, $x_{i\mtype}=3$ for each $i\in \{2,3,\cdots,n_{\ell,\delta}+1\}$ and the rest of the entries are equal to zero.
	We verify next in what follows that $x\in \calB_{\delta}(\calP,\calJ,\phi)$, and furthermore, $x$ is unique.
	The matrix $x$ satisfies, by construction, the row and column marginals.
	Since the rows in $\{2,3,n_{\ell,\delta}+1\}$ of $\calP$ are all equal, and the same holds for $x$, it is enough to show that the inequalities of Theorem~\ref{thm:cycles} are satisfied for the rows one and two.
	Observe that
	\begin{align*}
		\frac{\delta(x_{1\ftype})}{\calP_{1\ftype}}\cdot \frac{\delta(x_{2\mtype})}{\calP_{2\mtype}}- \frac{\delta(x_{1\mtype}+1)}{\calP_{1\mtype}}\cdot \frac{\delta(x_{2\ftype}+1)}{\calP_{2\ftype}}&=\frac{\delta(7+\ell)}{7}\cdot \frac{\delta(3)}{3-y_{\ell,\delta}}-\frac{\delta(1)^2}{\ell y_{\ell,\delta}}=\Gamma_{\ell,\delta}(\ell/n_{\ell,\delta})<0,
	\end{align*}
	and the other set of inequalities is immediately satisfied strictly since $x_{2\ftype}=x_{1\mtype}=0$ and $\delta(0)=0$.
	By Theorem~\ref{thm:cycles} we conclude that $x\in \calB_{\delta}(\calP,\calJ,\phi)=\calB_{\delta}(\calD_{\ell,\delta})$ and since the inequalities are satisfied strictly we have that $x$ is the unique $\delta$-biproportional solution of the instance $\calD_{\ell,\delta}$.
	
	Suppose now that $\delta(1)=0$. 
	Consider the matrix $\calP$ with $\ell+2$ rows and two columns defined as follows: $\calP_{1\ftype}=\ell+1$ and $\calP_{1\mtype}=1$; $\calP_{i\ftype}=1/(\ell+1)$ and $\calP_{i\mtype}=3-1/(\ell+1)$ for every $i\in \{2,3,\ldots,\ell+2\}$.
	We define the row marginals $\calJ$ such that $\calJ_1=\ell+2$ and $\calJ_i=3$ for every $i\in \{2,3,\ldots,\ell+2\}$, and the type marginals $\phi$ such that $\phi_{\ftype}=\ell+2$ and $\phi_{\mtype}=3\ell+3$.
	By construction the type marginals are integral.
	Consider in this case the instance $\calD_{\ell,\delta}=((\ell+1)\calP,\calJ,\phi)$.
	Since the matrix $\calP$ is such that $\sum_{i\in [n_{\ell,\delta}+1]}\calP_{it}=\phi_t$ for each $t\in \{\ftype,\mtype\}$ and $\calP_{i\ftype}+\calP_{i\mtype}=\calJ_i$ for each $i\in [n_{\ell,\delta}+1]$, by Proposition~\ref{prop:already-fair} we have that $\calP$ is the fair share of $(n_{\ell,\delta}\calP,\calJ,\phi)$.
	Furthermore, by Lemma \ref{lem:technical-biprop} \ref{opt:scaling} we have that $\calB_{\delta}((\ell+1) \calP,\calJ,\phi)=\calB_{\delta}(\calP,\calJ,\phi)$.
	Therefore, it is sufficient to compare the biproportional solutions of $(\calP,\calJ,\phi)$ with respect to $\calP$.
	\begin{figure}[H]
		$$\calP=\begin{pmatrix} \ell+1 & 1 \\ 1/(\ell+1) & 3-1/(\ell+1) \\ \vdots & \vdots \\ 1/(\ell+1) & 3-1/(\ell+1)\end{pmatrix}\quad \quad \calJ=\begin{pmatrix} \ell+2 \\ 3\\ \vdots\\ 3\end{pmatrix}\quad \quad \phi=\begin{pmatrix}\ell+2\\ 3\ell+3 \end{pmatrix}\quad \quad x=\begin{pmatrix} \ell+1 & 1 \\ 1 & 2 \\ \vdots & \vdots \\ 1 & 2\end{pmatrix}$$
		\caption{Two-dimensional instance and its biproportional solution when $\delta(1)=0$.}
	\end{figure}
	\noindent Consider the matrix $x$ defined as follows: $x_{i\ftype}=1$ for each $i\in \{1,2,\cdots,n_{\ell,\delta}+1\}$, $x_{i\mtype}=\ell+1$, $x_{i\mtype}=2$ for each $i\in \{2,3,\cdots,n_{\ell,\delta}+1\}$ and the rest of the entries are equal to zero.
	We verify next in what follows that $x\in \calB_{\delta}(\calP,\calJ,\phi)$, and furthermore, $x$ is unique.
	The matrix $x$ satisfies by construction the row and column marginals.
	Since the rows in $\{2,3,\ell+2\}$ of $\calP$ are all equal, and the same holds for $x$, it is enough to show that the inequalities of Theorem~\ref{thm:cycles} are satisfied for the rows one and two.
	Observe that
	\begin{align*}
		\frac{\delta(x_{1\ftype})}{\calP_{1\ftype}}\cdot \frac{\delta(x_{2\mtype})}{\calP_{2\mtype}}- \frac{\delta(x_{1\mtype}+1)}{\calP_{1\mtype}}\cdot \frac{\delta(x_{2\ftype}+1)}{\calP_{2\ftype}}&=\frac{\delta(1)}{\ell+1}\cdot \frac{\delta(3)}{3-1/(\ell+1)}-\frac{\delta(\ell+2)}{1}\cdot \frac{\delta(2)}{1/(\ell+1)}\\
		&=-(\ell+1)\cdot \delta(\ell+2)\cdot \delta(2)<0,
	\end{align*}
	\begin{align*}
		\frac{\delta(x_{2\ftype})}{\calP_{2\ftype}}\cdot \frac{\delta(x_{1\mtype})}{\calP_{1\mtype}}- \frac{\delta(x_{1\ftype}+1)}{\calP_{1\ftype}}\cdot \frac{\delta(x_{2\mtype}+1)}{\calP_{2\mtype}}&=\frac{\delta(1)}{1/(\ell+1)}\cdot \frac{\delta(\ell+1)}{1}-\frac{\delta(2)}{\ell+1}\cdot \frac{\delta(3)}{3-1/(\ell+1)}\\
		&=-\frac{\delta(2)\cdot \delta(3)}{3\ell+2}<0,
	\end{align*}
	By Theorem~\ref{thm:cycles} we conclude that $x\in \calB_{\delta}(\calP,\calJ,\phi)=\calB_{\delta}(\calD_{\ell,\delta})$ and since the inequalities are satisfied strictly we have that $x$ is the unique $\delta$-biproportional solution of $\calD_{\ell,\delta}$.
\end{proof}

\subsection{Fraction of Rows Violating the Fair Share Rounding}

Formally, given a two-dimensional instance $(\calP,\calJ,\phi)$ with fair share $\calF$ and given a signpost sequence $\delta$, let $\Lambda^{\delta}$ be the fraction of rows for which a biproportional solution $x\in \calB_{\delta}(\calP,\calJ,\phi)$ does not respect the fair share rounded up or down respectively, that is, $\Lambda^{\delta}(x,\calP,\calJ,\phi)=\frac{1}{n}|\{(i,t):x_{it}>\calF_{it}\}|$.
Observe that in an instance with two columns, if in a row one of two entries is smaller (larger) than the fair share rounded down (up), then the other entry of the same row is necessarilly larger (smaller) that the fair share rounded up (down).
That is, whenever there is a rounding violation, it happens to both entries of the row. 
The following is our main result in this line.
\begin{theorem}
	\label{thm:violates-fair}
	Let $\delta$ be a signpost sequence.
	Then, the following holds:
	\begin{enumerate}[label=(\alph*)]
		\item When $\delta(1)>0$, there exists a two-dimensional instance $\calT_{\delta}$ for which there is a unique $\delta$-biproportional solution $x$ of $\calT_{\delta}$ and $\Lambda^{\delta}(x,\calT_{\delta})\ge 1/(1+\lceil \delta(1)^{-2}+1\rceil)$.
		\item When $\delta(1)=0$, there exists a two-dimensional instance $\calT_{\delta}$ for which there is a unique $\delta$-biproportional solution $x$ of $\calT_{\delta}$ and $\Lambda^{\delta}(x,\calT_{\delta})\ge 1/3$.
	\end{enumerate}
	Furthermore, these instances have type marginals satisfying $\phi_{\ftype}=\phi_{\mtype}$.
\end{theorem}
That is, the theorem states that even when we restrict to instances where $\phi_{\ftype}=\phi_{\mtype}$ for the types, the fraction of rows that violate the fair share rounding is in general bounded away from zero, by a strictly positive constant.
In particular, when $\delta(1)=1$, which is the case for the Jefferson rounding, Theorem~\ref{thm:violates-fair} states the existence of a two-dimensional instance with only three rows for which the unique biproportional solution violates the fair share rounding in one row.
Same holds for other classic rounding with $\delta(1)=0$, the Adams rounding.

\begin{proof}[Proof of Theorem~\ref{thm:violates-fair}]
	Let $\delta$ be a signpost sequence such that $\delta(1)>0$ and let $n_{\delta}=\lceil \delta(1)^{-2}+1\rceil$.
	Consider the matrix $\calP$ with $1+n_{\delta}$ rows defined as follows: $\calP_{1\ftype}=3n_{\delta}-1$, 
	$\calP_{1\mtype}=1$, 
	$\calP_{i\ftype}=1/n_{\delta}$ for each $i\in \{2,3,\ldots,n_{\delta}+1\}$ and
	$\calP_{i\mtype}=3-1/n_{\delta}$ for each $i\in \{2,3,\ldots,n_{\delta}+1\}$.
	We define the marginals as follows: $\calJ^{\delta}_{1}=3n_{\delta}$ and $\calJ_{i}=3$ for each $i\in \{2,3,\ldots,n_{\delta}+1\}$, and the type marginals are given by $\phi_{\ftype}=\phi_{\mtype}=3n_{\delta}$.
	Consider the instance $\calT_{\delta}=(n_{\delta}\calP,\calJ,\phi)$.
	
	By construction, we have that $\sum_{i=1}^{n_{\delta}+1}\calP_{i\ftype}=\phi_{\ftype}$, $\sum_{i=1}^{n_{\delta}+1}\calP_{i\mtype}=\phi_{\mtype}$ and $\calP_{i\ftype}+\calP_{i\mtype}=\calJ_{i}$ for each $i\in [n_{\delta}+1]$.
	Therefore, by Propoposition~\ref{prop:already-fair} we have that $\calP$ is the fair share of $(n_{\delta}\calP,\calJ,\phi)$.
	Furthermore, by Lemma \ref{lem:technical-biprop} \ref{opt:scaling} we have that $\calB_{\delta}(n_{\delta} \calP,\calJ,\phi)=\calB_{\delta}(\calP,\calJ,\phi)$.
	Therefore, it is sufficient to compare the biproportional solutions of $(\calP,\calJ,\phi)$ with respect to $\calP$.
	Consider the matrix $x$ defined as follows: $x_{1\ftype}=3n_{\delta}$, $x_{1\mtype}=0$, $x_{i\ftype}=0$ for each $i\in \{2,3,\ldots,n_{\delta}+1\}$ and $x_{i\mtype}=3$ for each for each $i\in \{2,3,\ldots,n_{\delta}+1\}$.
	In what follows we prove that $x\in \calB_{\delta}(\calP,\calJ,\phi)$.
	From here the theorem follows since $\Lambda^{\delta}(x,\calT_{\delta})=1/(n_{\delta}+1)=1/(1+\lceil \delta(1)^{-2}+1\rceil)$.
	\begin{figure}[H]
		$$\calP=\begin{pmatrix} 3n_{\delta}-1 & 1\\ 1/n_{\delta} & 3-1/n_{\delta}\\ \vdots & \vdots \\ 1/n_{\delta} & 3-1/n_{\delta}\end{pmatrix}\quad \quad \calJ=\begin{pmatrix} 3n_{\delta} \\ 3\\ \vdots\\ 3\end{pmatrix}\quad \quad \phi=\begin{pmatrix}  3n_{\delta}\\ 3n_{\delta}\end{pmatrix} \quad \quad x=\begin{pmatrix} 3n_{\delta} & 0\\ 0 & 3\\ \vdots & \vdots \\ 0 & 3\end{pmatrix}$$
		\caption{Two-dimensional instance and its biproportional solution when $\delta(1)>0$.}
	\end{figure}
	\noindent To check that $x\in \calB_{\delta}(\calP,\calJ,\phi)$ it is sufficient to verify that the inequalities of Theorem~\ref{thm:cycles} are satisfied for $x$.
	Furthermore, by symmetry it is enough to check that the set of inequalities hold for the rows one and two.
	We have that the choice of $n_{\delta}$ guarantees that
	\begin{align*}
		\frac{\delta(x_{1\ftype})}{\calP_{1\ftype}}\cdot \frac{\delta(x_{2\mtype})}{\calP_{2\mtype}}-\frac{\delta(x_{2\ftype}+1)}{\calP_{2\ftype}}\cdot \frac{\delta(x_{1\mtype}+1)}{\calP_{1\mtype}}&=
		\frac{\delta(3n_{\delta})}{3n_{\delta}-1}\cdot \frac{\delta(3)}{3-1/n_{\delta}}-\frac{\delta(1)}{1/n_{\delta}}\cdot \frac{\delta(1)}{1}\\
		&\le \frac{3n_{\delta}}{3n_{\delta}-1}\cdot \frac{3}{3-1/n_{\delta}}-n_{\delta}\cdot \delta(1)^2\\
		&=n_{\delta}\left(\frac{9n_{\delta}}{(3n_{\delta}-1)^2}-\delta(1)^2\right)< 0,
	\end{align*}
	where in the second inequality we used that for every $x\ge 2$ it holds that $1/(x-1)>9x/(3x-1)^2$, and the last inequality holds by the definition of $n_{\delta}$.
	On the other hand, for the second inequality we have that 
	\begin{align*}
		\frac{\delta(x_{2\ftype})}{\calP_{2\ftype}}\cdot \frac{\delta(x_{1\mtype})}{\calP_{1\mtype}}- \frac{\delta(x_{1\ftype}+1)}{\calP_{1\ftype}}\cdot \frac{\delta(x_{2\mtype}+1)}{\calP_{2\mtype}}
		&=\frac{\delta(0)}{1/n_{\delta}}\cdot \frac{\delta(0)}{1}-\frac{\delta(3n_{\delta}+1)}{3n_{\delta}-1}\cdot \frac{\delta(4)}{3-1/n_{\delta}}\\
		&=-\frac{\delta(3n_{\delta}+1)}{3n_{\delta}-1}\cdot \frac{\delta(4)}{3-1/n_{\delta}}<0,
	\end{align*}
	since $\delta(0)=0$.
	The inequalities are satisfied strictly, therefore $x$ is the unique $\delta$-biproportional solution of $\calT_{\delta}$.
	That concludes the proof for the case when $\delta(1)>0$.
	
	Now suppose that $\delta(1)=0$ and let $\varepsilon_{\delta}\in (0,1)$ be any rational value such that  
	$$\frac{3\delta(6)-14\delta(5)}{3\delta(6)+7\delta(5)}\le \varepsilon_{\delta}\le \frac{21\delta(5)}{3\delta(6)+7\delta(5)}.$$
	We remark that this value exists, since $\delta(6)\le 6$ and $\delta(5)\ge 5$ and therefore $3\delta(6)-14\delta(5)<21\delta(5)$.
	Consider the matrix $\calP$ with three rows and two columns defined as follows: $\calP_{1\ftype}=7$, 
	$\calP_{1\mtype}=3$, 
	$\calP_{2\ftype}=\varepsilon_{\delta}$, 
	$\calP_{2\mtype}=3-\varepsilon_{\delta}$,
	$\calP_{3\ftype}=1-\varepsilon_{\delta}$ and 
	$\calP_{4\mtype}=2+\varepsilon_{\delta}$.
	We define the marginals as follows: $\calJ_{1}=10$ and $\calJ_{2}=\calJ_{3}=3$, and the type marginals are given by $\phi_{\ftype}=\phi_{\mtype}=8$.
	Consider the instance $\calT_{\delta}=(\alpha_{\delta}\calP,\calJ,\phi)$ where $\alpha_{\delta}$ is the smallest positive integer number such that $\alpha_{\delta}\varepsilon_{\delta}$ is integer.
	By construction we have that $\sum_{i=1}^{3}\calP_{i\ftype}=\phi_{\ftype}$, $\sum_{i=1}^{3}\calP_{i\mtype}=\phi_{\mtype}$ and $\calP_{i\ftype}+\calP_{i\mtype}=\calJ_{i}$ for each $i\in [3]$.
	Therefore, by Propoposition~\ref{prop:already-fair} we have that $\calP$ is the fair share of $(\calP,\calJ,\phi)$.
	Furthermore, by Lemma \ref{lem:technical-biprop} \ref{opt:scaling} we have that $\calB_{\delta}(\alpha_{\delta} \calP,\calJ,\phi)=\calB_{\delta}(\calP,\calJ,\phi)$.
	Therefore, it is sufficient to compare the biproportional solutions of $(\calP,\calJ,\phi)$ with respect to $\calP$.
	Consider the matrix $x$ defined as follows: $x_{1\ftype}=6$, $x_{1\mtype}=4$, $x_{2\ftype}=x_{3\ftype}=1$ and $x_{2\mtype}=x_{3\mtype}=2$.
	In what follows we prove that $x\in \calB_{\delta}(\calP,\calJ,\phi)$.
	From here the theorem follows since $\Lambda^{\delta}(x,\calT_{\delta})=1/3$.
	\begin{figure}[H]
		$$\calP=\begin{pmatrix} 7 & 3\\ \varepsilon_{\delta} & 3-\varepsilon_{\delta}\\  1-\varepsilon_{\delta} & 2+\varepsilon_{\delta}\end{pmatrix}\quad \quad \calJ=\begin{pmatrix} 10 \\ 3\\ 3\end{pmatrix}\quad \quad \phi=\begin{pmatrix}  8\\ 8\end{pmatrix} \quad \quad x=\begin{pmatrix} 6 & 4 \\ 1 & 2 \\1 & 2\end{pmatrix}$$
		\caption{Two-dimensional instance and its biproportional solution when $\delta(1)=0$.}
	\end{figure}
	\noindent To check that $x\in \calB_{\delta}(\calP,\calJ,\phi)$ it is sufficient to verify that the inequalities of Theorem~\ref{thm:cycles} are satisfied for $x$.
	We have that the inequalities for rows two and three are satisfied directly since $x_{2\ftype}=x_{3\ftype}=1$, $x_{2\mtype}=x_{3\mtype}>1$ and $\delta(1)=0$.
	Consider the rows one and two.
	The choice of $\varepsilon_{\delta}$ guarantee that 
	\begin{align*}
		\frac{\delta(x_{1\ftype})}{\calP_{1\ftype}}\cdot \frac{\delta(x_{2\mtype})}{\calP_{2\mtype}}-\frac{\delta(x_{2\ftype}+1)}{\calP_{2\ftype}}\cdot \frac{\delta(x_{1\mtype}+1)}{\calP_{1\mtype}}=
		\frac{\delta(6)}{7}\cdot \frac{\delta(2)}{3-\varepsilon_{\delta}}-\frac{\delta(2)}{\varepsilon_{\delta}}\cdot \frac{\delta(5)}{3}<0.
	\end{align*}
	On the other hand, for the second inequality we have that 
	\begin{align*}
		\frac{\delta(x_{2\ftype})}{\calP_{2\ftype}}\cdot \frac{\delta(x_{1\mtype})}{\calP_{1\mtype}}- \frac{\delta(x_{1\ftype}+1)}{\calP_{1\ftype}}\cdot \frac{\delta(x_{2\mtype}+1)}{\calP_{2\mtype}}&=\frac{\delta(1)}{\varepsilon_{\delta}}\cdot \frac{\delta(4)}{3}-\frac{\delta(8)}{7}\cdot \frac{\delta(4)}{3-\varepsilon_{\delta}}\\
		&=-\frac{\delta(8)}{7}\cdot \frac{\delta(4)}{3-\varepsilon_{\delta}}<0,
	\end{align*}
	since $\delta(0)=0$.
	Now consider rows one and three.
	The choice of $\varepsilon_{\delta}$ guarantee that 
	\begin{align*}
		\frac{\delta(x_{1\ftype})}{\calP_{1\ftype}}\cdot \frac{\delta(x_{3\mtype})}{\calP_{3\mtype}}-\frac{\delta(x_{3\ftype}+1)}{\calP_{3\ftype}}\cdot \frac{\delta(x_{1\mtype}+1)}{\calP_{1\mtype}}=
		\frac{\delta(6)}{7}\cdot \frac{\delta(2)}{2+\varepsilon_{\delta}}-\frac{\delta(2)}{1-\varepsilon_{\delta}}\cdot \frac{\delta(5)}{3}<0.
	\end{align*}
	On the other hand, for the second inequality we have that 
	\begin{align*}
		\frac{\delta(x_{3\ftype})}{\calP_{3\ftype}}\cdot \frac{\delta(x_{1\mtype})}{\calP_{1\mtype}}- \frac{\delta(x_{1\ftype}+1)}{\calP_{1\ftype}}\cdot \frac{\delta(x_{3\mtype}+1)}{\calP_{3\mtype}}&=\frac{\delta(1)}{1-\varepsilon_{\delta}}\cdot \frac{\delta(4)}{3}-\frac{\delta(7)}{7}\cdot \frac{\delta(3)}{2+\varepsilon_{\delta}}\\
		&=-\frac{\delta(7)}{7}\cdot \frac{\delta(3)}{2+\varepsilon_{\delta}}<0,
	\end{align*}
	since $\delta(0)=0$.
	Since the inequalities are satisfied strictly, $x$ is the unique $\delta$-biproportional solution of $\calT_{\delta}$.
	That concludes the proof for the case when $\delta(1)>0$.
\end{proof}

\section{Consequences on the Quality of Apportionment}
\label{sec:consequences}
In the setting of biproportional apportionment by Balinski and Demange, a two-dimensional instance corresponds to a triplet $(V,R,C)$ where $V$ is a matrix of dimensions $p\times d$, $R$ is a vector in $\ZZ^p$ and $C$ is a vector in $\ZZ^d$ such that $\sum_{i\in [p]}R_i=\sum_{j\in [d]}C_j$.
We say that a set valued function $\varphi$ is a {\it two-dimensional apportionment method} if for every two-dimensional instance $(V,R,C)$ we have that every $x\in \varphi(V,R,C)$ is a non-negative and integral matrix of dimensions $p\times d$ that satisfies the following: $\sum_{j\in [d]}x_{ij}=R_i$ for every $i\in [p]$ and $\sum_{i\in [p]}x_{ij}=C_j$ for every $j\in [d]$.
Given a signpost sequence $\delta$, a $\delta$-biproportional solution of an instance $(V,R,D)$ is defined exactly by the conditions \eqref{biprop1}-\eqref{biprop4} for every $i\in [p]$ and every $t\in [d]$, and where $\calS_{it}=h$ for every $i\in [p]$ and every $t\in [d]$.

Balinski and Demange proved that the family of biproportional methods are the unique two-dimensional apportionment methods that satisfy a list of natural properties (exactness, monotonicity, uniformity) 
that we call the {\bf BD} properties \cite{BalinskiDemange1989}[Section II, p. 711].
The fair share of an instance $(V,R,D)$ is defined exactly by the conditions \eqref{fs1}-\eqref{fs3} for every $i\in [p]$ and every $t\in [d]$.
We say that a two-dimensional method satisfies the {\it lower quota} property if the output of the method is always at least the fair share rounded down.
Similarly, a two-dimensional method satisfies the {\it upper quota} property if the output of the method is always at most the fair share rounded up.
Our results from Section \ref{sec:fairness} for the case of $d=2$, together with the characterization of Balinski and Demange, imply the following impossibility result.
\begin{corollary}
	\label{cor:impossibility}
	There is no two-dimensional apportionment method that simultaneously satisfy the {\bf BD} properties and the lower quota property.
	Similarly, there is no two-dimensional apportionment method that simultaneously satisfy the {\bf BD} properties and the upper quota property.
\end{corollary} 
This constrasts with the one dimensional case $d=1$ where the {\bf BD} properties are compatible with the lower quota or upper quota properties.
The Jefferson method is the unique divisor method that satisfies the lower quota property, while Adams method is the unique divisor method that satisfies the upper quota property \cite{BYbook}.
Balinski and Young~\cite{BYbook} studied the stronger property where both lower and upper quota are satisfied simultaneously, known as the {\it staying within fair share} property. 
They studied this stronger property in the context of the one dimensional ($d=1$) apportionment problem, where the fair share in that case corresponds to the fractional assignment obtained by assigning seats proportionally to the votes obtained for each party. 
Surprisingly, in this case there is a strong impossibility result: Balinski and Young showed that there is no divisor method staying within fair share~\cite[Corollary 6.1, p.130]{BYbook}.
As a corollary of Theorems~\ref{thm:two-rows} and \ref{thm:violates-fair} we get the following corollary.

\begin{corollary}
	\label{thm:biprop-no-breach}
	For every signpost sequence $\delta$ with $\delta(1)>0$, the $\delta$-biproportional parity mechanism stays within fair share for instances satisfying the supply condition with only two parties.
	On the other hand, for every signpost sequence $\delta$, the $\delta$-biproportional parity mechanism does not stay within fair share for instances with three or more parties.
\end{corollary}

For the greedy \& parity correction mechanism, the picture can be in principle worse than in the biproportional setting, and this would not be surprising since the main goal of this mechanism is not to achieve biproportionality.
However, we show that the fair share rounding violation of the solution obtained by the greedy \& parity correction mechanism can be in general very high.
Given a valid apportionment mechanism $\calM$, given a signpost sequence $\gamma$ and a solution $\calE\in \calM(\calI,\gamma)$, we define the two-dimensional instance $(\calP(\calI),\calJ(\calE),\phi(\calE))$ such that
for each party $i\in [n]$ and each type $t\in \{\ftype,\mtype\}$ we have $\calJ_i(\calE)=\sum_{c\in \calC_i}\calE(c)$ and $\phi_t(\calE)=\sum_{c\in \calC^t}\calE(c)$.
We denote by $\calF(\calE)$ the fair share of the instance $(\calP(\calI),\calJ(\calE),\phi(\calE))$.
\begin{theorem}
	\label{thm:chilean-breach} 
	There exists an instance $\calI$ with two parties that satisfies the supply condition and such that for every signpost sequence $\gamma$ and every $\calE\in \calM^{G}(\calI,\gamma)$ we have $\sum_{c\in \calC_i^t}\calE(c)\notin \{\lfloor \calF_{it}(\calE)\rfloor,\lceil \calF_{it}(\calE) \rceil\}$
	for each $i\in \{1,2\}$ and each type $t\in \{\ftype,\mtype\}$.
	In particular, the greedy \& parity correction mechanism does not stay within fair share. 
\end{theorem}

This contrasts with the positive result in Corollary~\ref{thm:biprop-no-breach} for the biproportional parity mechanism that guarantees to stay within fair share for instances with only two parties.
In what follows we provide the details on the construction of the family of instances in order to show Theorem~\ref{thm:chilean-breach}.

Consider the instance $\calI$ defined as follows: we have two parties and each party has six candidates. 
The house size is equal to $h=6$.
For each party we have six candidates of each type, that is, $C_i^{\ftype}=\{c_{i,1},c_{i,2},c_{i,3}\}$ and $C_i^{\mtype}=\{c_{i,4},c_{i,5},c_{i,6}\}$, with $i\in \{1,2\}$.
The votes for the candidates are defined as follows: We have 
$\votes_{\calI}(c_{1,j})=345$ for $j\in \{1,2,3\}$ and 
$\votes_{\calI}(c_{1,j})=55$ for $j\in \{4,5,6\}$;
$\votes_{\calI}(c_{2,j})=184$ for $j\in \{1,2,3\}$ and
$\votes_{\calI}(c_{2,j})=16$ for $j\in \{4,5,6\}$.
Since the house size is even we have that $\phi_{\ftype}(\calI)=\phi_{\mtype}(\calI)=3$.
We remark that $\calI$ satisfies the supply condition, and therefore the greedy \& parity correction mechanism~\ref{alg:chilean} is guaranteed to terminate with a feasible allocation thanks to Lemma~\ref{lem:chilean-correct}.
\begin{lemma}
	\label{lem:chilean-LB-output}
	Let $\calE$ be the output of Algorithm~\ref{alg:chilean} in instance $\calI$. 
	Then, the following holds:
	$\sum_{c\in \calC_1^{\ftype}}\calE(c)=3$, $\sum_{c\in \calC_1^{\mtype}}\calE(c)=1$,
	$\sum_{c\in \calC_2^{\ftype}}\calE(c)=0$ and $\sum_{c\in \calC_2^{\mtype}}\calE(c)=2$.
\end{lemma}

\begin{proof}
	We first show that for every signpost sequence $\gamma$ we have $\apportiong(\calQ(\calI))=\{(4,2)\}$.
	We have that $\calQ_1(\calI_n)=342\cdot 3+52\cdot 3+3\cdot 6=1200$, and  $\calQ_2(\calI_n)=181\cdot 3+13\cdot 3+3\cdot 6=600$.
	Therefore, it follows by Lemma \ref{lem:jefferson-one-dim} \ref{jeff-proportional}: The unique solution is given by taking $\lambda=300$, $\calQ_i(\calI_n)/\lambda\in \ZZ$ for each $i\in \{1,2\}$ and $\calQ_1(\calI)/\lambda+\calQ_2(\calI)/\lambda=4+2=6$.
	In Phase 1 the algorithm selects the top four candidates of party 1, and the top two candidates of party 2.
	The top four candidates of party 1 are given by $c_{1,1},c_{1,2},c_{1,3}$ of type $\ftype$ and one candidate $c_{1,j_1}$ with $j_1\in \{4,5,6\}$ of type $\mtype$.
	The top two candidates of party 2 are $c_{2,k_2}$ and $c_{2,\ell_2}$ of type $\ftype$ with $k_2,\ell_2\in \{1,2,3\}$.
	Therefore, at the end of Phase 1, the algorithm has selected a total of five candidates of type $\ftype$ and one candidates of type $\mtype$, that is, the overrepresented type is $t^{\star}=\ftype$ and $t_{\star}=\mtype$.
	
	From the set of selected candidates of type $\ftype$, we have that exactly two of them received 184 votes, while three of them received $345$ votes.
	The selected candidates that received 184 votes are exactly those given by $c_{2,k_2}$ and $c_{2,\ell_2}$ of type $\ftype$ and therefore they are replaced by the top two candidates of type $\mtype$ of the same party, which in this case corresponds to two candidates $c_{2,r_2}$ and $c_{2,s_2}$ with $r_2,s_2\in \{4,5,6\}$ that received 16 votes each.
	Parity has been achieved and the algorithm terminates with the following output: 
	$\{c\in \calC_1^{\ftype}:\calE(c)=1\}=\{c_{1,1},c_{1,2},c_{1,3}\}$, $\{c\in \calC_1^{\mtype}:\calE(c)=1\}=\{c_{1,j_1}\}$, $\{c\in \calC_2^{\ftype}:\calE(c)=1\}=\emptyset$ and $\{c\in \calC_2^{\mtype}:\calE(c)=1\}=\{c_{2,r_2},c_{2,s_2}\}$.
	That concludes the proof.
\end{proof}

\begin{proof}[Proof of Theorem~\ref{thm:chilean-breach}]
Let $\lambda_{1}=1/25$ and $\lambda_{2}=1/20$,
	and let $\mu_{\ftype}=1/23$ and $\mu_{\mtype}=1/3$ be the multipliers associated to the types.
	Let $\calF\in \RR^{2\times 2}$ be the matrix such that for each $i\in \{1,2\}$ and each type $t\in \{\ftype,\mtype\}$ we have $\calF_{it}=\lambda_{i}\mu_t\sum_{c\in \calC_{i}^t}\votes_{\calI}(c)$, that is,
	$\calF_{1\ftype}=1.8$, $\calF_{1\mtype}=2.2$,
	$\calF_{2\ftype}=1.2$ and $\calF_{2\mtype}=0.8$.
	Then, $(\calF,\lambda,\mu)$ satisfies conditions (\ref{fs1})-(\ref{fs3}) and the uniqueness of the fair share implies that $\calF$ is the fair share $\calF(\calE)$ of the two-dimensional instance $(\calP(\calI),\calJ(\calE),\phi(\calE))$.
	By Lemma~\ref{lem:chilean-LB-output}, we have $\sum_{c\in \calC_1^{\ftype}}\calE(c)>\lceil \calF_{1\ftype}\rceil$, $\sum_{c\in \calC_1^{\mtype}}\calE(c)<\lfloor \calF_{1\mtype}\rfloor$, $\sum_{c\in \calC_2^{\ftype}}\calE(c)<\lfloor \calF_{2\ftype}\rfloor$ and $\sum_{c\in \calC_2^{\mtype}}\calE(c)>\lceil \calF_{2\ftype}\rceil$.
\end{proof}
\newpage
\bibliographystyle{abbrv}
{\small \bibliography{references}}

\section{Appendix}

\begin{proof}[Proof of Lemma \ref{lem:jefferson-one-dim}]
Properties \ref{jeff-proportional}-\ref{jeff-scaling} come directly from the definition of a divisor method.
We now prove \ref{jeff-monotone}.
Suppose there exist $\calJ\in \apportiong(\calQ,h)$ and $\calJ'\in \apportiong(\calQ',h)$ such that $\calJ_{\pparty}< \calJ'_{\pparty}$. 
By the population monotonicity property \cite{BYbook}[Appendix A, p. 117], since $\calQ'_p/\calQ'_i>\calQ_p/\calQ_i$ for every $i\ne \pparty$ and $\calJ_{\pparty}< \calJ'_{\pparty}$, it holds that $\calJ'\le \calJ_i$ for every $i\ne \pparty$.
But then we have that $\sum_{j\in [n]}\calJ'_j<\sum_{j\in [n]}\calJ_j$, which is contradiction since the house size is equal in both instances. 
That concludes \ref{jeff-monotone}.

We now prove \ref{jeff-monotone-strong}.
Consider the instance obtained from $(\calQ,h)$ and any $\calJ\in \apportiong(\calQ,h)$ as follows: The parties set is $[n]\setminus \{\pparty\}$, the votes are given by $\calT_j=\calQ_j$ for each $j\ne \pparty$ and the house size is $h-\calJ_{\pparty}$.
In particular, the restriction of $\calJ$ to the parties in $[n]\setminus \{\pparty\}$ belongs to $\apportiong(\calT,h-\calJ_{\pparty})$.
Similarly, consider the instance obtained from $(\calQ',h)$ and any $\calJ'\in \apportiong(\calQ',h)$ as follows: The parties set is $[n]\setminus \{\pparty\}$, the votes are given by $\calT_j=\calQ'_j=\calQ_j$ for each $j\ne \pparty$ and the house size is $h-\calJ'_{\pparty}$.
In particular, the restriction of $\calJ'$ to the parties in $[n]\setminus \{\pparty\}$ belongs to $\apportiong(\calT,h-\calJ'_{\pparty})$.

Observe that by property \ref{jeff-monotone} we have that $\calJ_{\pparty}\ge \calJ'_{\pparty}$ for every $\calJ\in \apportiong(\calQ,h)$ and $\calJ'\in \apportiong(\calQ',h)$.
Fix a solution $\calJ\in \apportiong(\calQ,h)$ and fix a solution $\calJ'\in \apportiong(\calQ',h)$.
Define the vector $\calG$ with entries in $[n]\setminus \{\pparty\}$ such that $\calG_j=\calJ_j$ for every $j\ne \pparty$.
In particular, $\calG$ belongs to $\apportiong(\calT,h-\calJ_{\pparty})$.
By the house monotonicity property \cite{BYbook}[Appendix A, p. 117] the following holds: There exists $\calS\in \apportiong(\calT,h-\calJ'_{\pparty})$ such that $\calJ_j=\calG_{j}\le \calS_{j}$ for every $j\ne \pparty$.
For any such solution $\calS\in \apportiong(\calT,h-\calJ'_{\pparty})$ consider the vector $\calH$ defined as follows: $\calH_{\pparty}=\calJ'_{\pparty}$ and $\calH_{j}=\calS_{j}$ for every $j\ne \pparty$.
By the uniformity property of divisor methods \cite{balinski1980webster,BalinskiDemange1989} we have that $\calH\in \apportiong(\calQ',h)$ and it satisfies that $\calH_j\ge \calJ_j$ for every $j\ne \pparty$.
This concludes \ref{jeff-monotone-strong}.
\end{proof}
\end{document}